\documentclass[12pt]{article}

\usepackage{latexsym}
\usepackage{amsfonts, amsmath}
\usepackage{amssymb}
\usepackage{theorem}
\usepackage[english]{babel}
\usepackage{longtable}
\usepackage[hidelinks]{hyperref}

\usepackage{etex}
\usepackage{tikz}
\usepackage{pgflibraryarrows}  
\usepackage{pgflibrarysnakes}   
\usetikzlibrary{arrows,shapes,spy,positioning,trees,backgrounds} 
\usepackage{pgfplots}
\pgfplotsset{width=7cm,compat=1.15}  
\usepgfplotslibrary{fillbetween}            

\usepackage{sectsty}
\allsectionsfont{\raggedright}

\usepackage{rotating}  

\usepackage{graphicx} 
\usepackage{subfigure} 

\usepackage{float}

\usepackage[margin=2cm]{geometry}

\newtheorem{theorem}{Theorem}[section]
\newtheorem{definition}[theorem]{Definition}

\newtheorem{proposition}[theorem]{Proposition}

\newtheorem{lemma}[theorem]{Lemma}

\newtheorem{example}[theorem]{Example}

\newenvironment{proof}
{\begin{trivlist}\item[]{Proof:}}{\hfill{$\square$}\noindent\end{trivlist}}

\usepackage{tikz}

\begin{document}
\title{Nash's bargaining problem and\\ the scale-invariant Hirsch  citation index}

\author{\normalsize{
Josep Freixas\footnote{Universitat Polit\`{e}cnica de Catalunya (Campus Manresa), Departament de Matem\`{a}tiques;
e-mail: josep.freixas@upc.edu;
postal address: EPSEM, Avda. Bases de Manresa, 61-73, E-08242 Manresa, Spain.}\ ,
Roger Hoerl\footnote{Union College, Department of Mathematics;
e-mails: [hoerlr,zwickerw]@union.edu;
postal address: Union College, 807 Union Street,
Schenectady,  NY 12308, USA.}\ ,
and William S. Zwicker$^{\dagger ,}$\footnote{Murat Sertel Center for Advanced Economic Studies, Istanbul Bilgi University, Turkey.}
}}

\maketitle

\begin{abstract}

A number of citation indices have been proposed for measuring and ranking the research publication records of scholars. Some of the best known indices, such as those proposed by Hirsch and Woeginger, are designed to reward most highly those records that strike some balance between productivity (number of papers published), and impact (frequency with which those papers are cited). A large number of rarely cited publications will not score well, nor will a very small number of heavily cited papers.

We discuss three new citation indices, one of which was independently proposed in \cite{FHLB}. Each rests on the notion of \emph{scale invariance}, fundamental to John Nash's solution of the two-person bargaining problem. Our main focus is on one of these---a scale invariant version of the Hirsch index.  We argue that it has advantages over the original; it produces fairer rankings within subdisciplines, is more decisive (discriminates more finely, yielding fewer ties) and more dynamic (growing over time via more frequent, smaller increments), and exhibits enhanced centrality and tail balancedness.
Simulations suggest that scale invariance improves robustness under Poisson noise, with  increased decisiveness having no cost in terms of the number of ``accidental" reversals, wherein random irregularities cause researcher $A$ to receive a lower index value than $B$, although $A$'s  productivity and impact are both slightly higher than $B$'s.


Moreover, we provide an axiomatic characterization of the scale invariant Hirsch index, via axioms that bear a close relationship, in discrete analogue, to those used by Nash in \cite{Nas50}.  This argues for the mathematical naturality of the new index.\footnote{An earlier version was presented at the $5^{th}$ World Congress of the Game Theory Society, Maastricht, Netherlands in 2016.}

%

\vskip 0.2truecm

\end{abstract}

\section{Introduction} \label{SEC:INTRO}

In the academic system, many crucial decisions concerning faculty recruitment, promotion, Ph.D. positions, awarding of grants, and research travel money depend on how  research is evaluated.  Often, these decisions compare researchers from similar fields and of similar scientific age. Several bibliometric measures based on productivity and impact have been proposed for measuring and ranking their research publication records. These measures are alternatives to other, simpler bibliometric indicators such as sum of all citations, average number of citations, and number of publications. Two of these alternative measures are Hirsch's $h$-index \cite{Hir05, Hir07}, which is the most used bibliometric measure today, and  Woeginger's $w$-index \cite{Woe08a, Woe08b}.

Each of these measures can be calculated from a scientist's  \emph{citation record}, which is the vector ${\bf x} = (x_1,x_2,\dots,x_{l({\bf x})})$ of positive integers in which $x_1$ is the number of citations of the scientist's most cited paper, $x_2$ is the number of citations of the second-most cited paper, etc., so that $x_1 \geq x_2 \geq \dots \geq x_{l({\bf x})}$; here ${\bf x}$'s \emph{length} $l({\bf x})$ is the number of publications that have been cited once at least.\footnote
{
It would do no harm to allow some $0$s (papers without citations) to appear in ${\bf x}$, but these typically do not affect the value of an index. If vectors ${\bf x}$ with $l({\bf x})=0$ were admitted it would mean that either the researcher has no publications yet, or no publications that have generated citations yet. In both cases any reasonable citation record would assign a value of zero to these vectors. In fact, this is an obvious requirement in the definition of a citation index on integers (see, e.g., \cite{Woe08a}). 
}
In particular, a scientist has a \emph{Hirsch $h$-index} of $h$ if $x_h \geq h$ and $x_{h+1} \leq h$---that is, `if $h$ of his or her papers have at least $h$ citations each and the other papers have less than or equal to $h$ citations each' \cite{Hir05}.
Today, the $h$-index is a widely used indicator of research output, computed automatically in the Web of Science (WoS, provided by Thomson Reuters, Philadelphia, PA, USA) and in other literature databases such as Scopus or Google Scholar. Axiomatizations of the $h$-index have been proposed in \cite{Woe08a, Woe08b} and \cite{Kon14}. A number of studies show that a scientist's $h$-index corresponds to peer judgements (\cite{BoDa05, BoDa07, BoDa09}) and thus has convergent validity.

Note that any citation record ${\bf x}$ can be represented as a bar-graph in the first quadrant (see Figure 1); each bar has width $1$, while the height of the $i^{\emph th}$ bar is $x_i$ (the number of citations of the $i^{\emph th}$-most cited paper).  This bar graph can be considered a region $\mathcal B({\bf x}) \subseteq {\mathbb{R}}^2$ in the first quadrant (the shaded area of Figure \ref{fig:curves}a, or that of Figure \ref{fig:BasicIndices}a including both the lighter and darker regions).  The sum $\sum x_i$ of all citations is now given by the total area of region $\mathcal B({\bf x})$, while the Hirsch index $h$ is given by the \emph{square root} of the area of the square sub-region---that is,  by the side-length of the  \emph{Hirsch square}, which is the square of maximal area, among all squares inscribed in $\mathcal B({\bf x})$ with one corner at the origin.
Why, then, did Hirsch choose to apply a square root to the square's area?  And why omit contributions from the \emph{remainder} of  $\mathcal B({\bf x})$---that part of $\mathcal B({\bf x})$ lying outside the Hirsch square? These questions go to the heart of Hirsch's reasons for proposing his index as an improvement over the simpler bibliometric indicators (such as the sum of all citations)
in use at the time, and we address them next.

The square root guarantees  \emph{linear growth} over time, for the Hirsch index of a single researcher, under the following \emph{simple deterministic model}: each researcher $R$ is endowed with a \emph{productivity parameter} $p$ and an \emph{impact parameter} $c$; if $p$ and $c$ are integers then $R$ publishes exactly $p$ papers each year of their career, with each paper attracting exactly $c$ citations in each year subsequent to, or equal to, the year of publication.\footnote
{\label{floorrounding}
In particular, at the end of year $1$, $p$ papers have been published and each paper has been cited in $c$ publications.  For non-integer parameter values, $\lfloor np \rfloor$ papers are published in the first $n$ years of a career, and each paper is cited $\lfloor kc \rfloor$ times over the first $k$ years subsequent to, or equal to, the year in which it was published.
}
 Let $h_R(n)$ denote $R$'s Hirsch index after $n$ years of $R$'s career.  Then under this model (with integer-valued parameters and for $n$ any positive integer) the points $(n, h_R(n))$ all lie in a strip between two closely spaced parallel lines of common slope $s_{h,R}$ given by


\begin{equation}  \label{EQ:slope-h}
s_{h,R} = \dfrac{pc}{p+c},
\end{equation}

\noindent as shown in \cite{Hir05}.\footnote{Hirsch's version is phrased somewhat differently.
The same is true when $p$ and $c$ are not integers, but we omit the detailed argument for this case in Section \ref{SEC:LinearGrowth}.
} Hirsch then argues that to compare the research records of two scholars of different ages, without ceding any automatic advantage to the one whose publishing career started earlier, we should compare their slopes.

The Hirsch square divides the remainder of $\mathcal B({\bf x})$ into two disjoint regions: the \emph{vertical tail} lies above the Hirsch square, and the \emph{horizontal tail} lies to its right. 
Truncating both tails thus rewards most those publication records that achieve a balance between productivity and impact. As we discuss in Section \ref{BalanceCost}, imposing a balance in this way has a cost; it requires sacrificing other properties that may be seen as desirable. 
We do find both the argument for linear growth, and that for balancing productivity and impact, to be compelling, 
 but take issue with the particular method used to impose that balance. By employing a \emph{square}, the Hirsch index equates a unit on the horizontal axis (a publication) with one on the vertical axis (a citation).  Others have also observed that the $h$-index suffers from this implicit reliance on comparability of scale between two axes.  ``The problem is that Hirsch assumes an equality between incommensurable
quantities \dots Hirsch's index \dots posits an equality between two
 quantities with no evident logical connection" \cite{LJL08}.
This equality ``is viewed as an oversimplification and as arbitrary" \cite{Ley09}.
In Section \ref{SEC:Axiomatization}, we will argue that this equating of units has consequences that compromise the value of the Hirsch index as a tool for comparing research records, distorting the ranking of scholars, even when they work in the same  subdiscipline  and are of similar scientific age.  There we point out, as well, that the argument for truncating the vertical and horizontal tails are not the same, so there is no reason to truncate these two tails at the same place,  as is done by the Hirsch square.\footnote{
In Hirsch's original paper \cite{Hir05} the top boundary of $\mathcal B({\bf x})$ is shown as a smooth curve, roughly the shape of the hyperbola $xy=1$, which is symmetric about the diagonal line $y = x$, so that the two tails have the same shape (after reflection) and in particular have the same area. As shown in \cite{FHLB}, this sort of symmetry is not typical for actual citation records. 
} Scale-invariant versions are free from these defects.


We first proposed scale-invariant versions of the Hirsch and Woeginger indices in a slide presentation at the $5^{th}$ World Congress of the Game Theory Society, in 2016.\footnote{See p10 on the \emph{programmes} link at https://project.dke.maastrichtuniversity.nl/games2016/programme.html. At that time, we had not yet found an axiomatic characterization having satisfactory normative content.} Independently, Fenner, Harris, Levene, and Bar-Ilan (see \cite{FHLB}) proposed their $\chi$-index, which is identical to the scale-invariant Hirsch index $h'$ we discuss here; their follow-up paper  \cite{LFB2019} provides an axiomatization.  In Section \ref{SEC:INTRO}.1 we discuss their contributions, which are largely complementary to our own.
Here we consider several scale-invariant citation indices, including alternative versions of the $h$ and $w$ indices, that are intended to factor out any presumption that a unit on one axis is comparable to a unit on the other. Our approach is to import a key idea from John Nash's solution to the \emph{two-person bargaining problem}. In Nash's context, a \emph{feasible set} is a closed and bounded convex region $F \subseteq {\mathbb{R}}^2$, and a point $(x_1, x_2) \in F$ represents a feasible bargain, in the form of utility payoffs $x_1$ for agent 1 and $x_2$ for agent 2.  A solution $\phi$ to the bargaining problem selects one such point $\phi (F) = (\phi_1 (F) , \phi_2 (F)) \in F$ from each feasible set $F$.  One common principle of economics is that a single agent can make internal comparisons of utility (a vacation in France yields twice as much utility to Sarah as a vacation in Alaska), but interpersonal comparisons of utility are not meaningful; one cannot say that Sarah gains more utility from a vacation in France than does Piotr.  This idea may be expressed mathematically by requiring that any individual's utility scale be defined only up to a linear rescaling $x \mapsto cx$, where $c$ is an arbitrary positive real number.  In Nash's context, this says that  for any two real constants $c_1, c_2 \geq 0$, if
$\psi \!\! : \mathbb{R}^2 \rightarrow \mathbb{R}^2$ is defined by $\psi(x_1 , x_2) = (c_1 x_1, c_2 x_2)$ then the stretched image $\psi [F]$ of a feasible set is not distinguishable from the original $F$, so that a solution $\phi$ should pick corresponding points from these two feasible sets: $\phi(\psi [F]) = \psi(\phi(F))$.  This requirement is Nash's \emph{scale invariance} axiom; it guarantees that the utility Sarah derives from the final bargain with Piotr is unaffected by any change Piotr might make to the size of the unit he uses to measure/report his own utility.

We'll presume that Nash's \emph{disagreement point} has been shifted to the origin, and impose some additional mild restrictions on the feasible region $F$: it lies in the first quadrant, and whenever $0 \leq y_1 \leq x_1$ and $0 \leq y_2 \leq x_2$ with $(x_1, x_2) \in F$, we have $(y_1, y_2) \in F$.\footnote{These conditions do not impose any significant loss of generality on Nash's theorem.} Nash demonstrates that the unique function $\phi$ satisfying his axioms chooses the point $(x_1,x_2)$ maximizing the product $x_1x_2$ of utilities.  This point is unique, thanks to convexity, and is the same as the point on the upper-right-hand corner of the rectangle of greatest area, among all rectangles that fit inside $F$ and have one corner at the origin.  In other words, scale invariance is achieved by using, in place of a square, a \emph{rectangle} of variable proportions.

 Importing Nash's idea, we set the value of the \emph{scale-invariant Hirsch index} $h'$ to be the \emph{square root} of the area of the rectangular sub-region of maximal area, among all rectangles inscribed in $\mathcal B({\bf x})$ with one corner at the origin.\footnote{
Unlike a feasible set for the Nash Bargaining Theorem, the $\mathcal B({\bf x})$ region is not generally convex, so the inscribed rectangle of greatest area need not be unique.  This does not seem to present any difficulties, as the index value depends only on the area of that rectangle, which \emph{is} unique.
} We provide a more precise definition later.  
Loosely speaking, the resulting scale invariance of $h'$ tells us that the \textit{relative} standing of two researchers $R$ and $T$---as represented, for example, by the ratio of their citation index values---would be unaffected by applying a common vertical (or horizontal) stretch (or compression) to both bar graphs $\mathcal B({\bf x}_R)$ and $\mathcal B({\bf x}_T)$.\footnote{Scale invariance is potentially at odds with balance. For example, if the rectangle of maximal area has minimal height and enormous width, reflecting a scholarly record with many papers that were cited by very few, this lack of balance might still be rewarded with a significant $h'$ value. We address this matter in the concluding Section \ref{SEC:MoreQuestions} by recommending that, in practice, rectangles with extreme proportions be disallowed. }

Although we discuss two other scale-invariant citation indices, our main focus in the article is with this version $h'$ of the Hirsch index, for which we provide an axiomatic characterization.
The five axioms resemble, in discrete version, those used by Nash for the two-person bargaining problem.
In fact, four of these five axioms are satisfied by all three of the scale invariant indices we propose.  These four pick out a class of scale-invariant citation indices, while the fifth axiom, \emph{Max-bounded}, picks    $h'$ out from this broader class and is closely related to the axiom of Nash's that has been called \emph{Independence of Irrelevant Alternatives} axiom, \emph{aka} ``IIA."\footnote{\label{Salles}Referring to Nash's axiom, Salles \cite{Salles23} states that, ``It has been rather unfortunate that
this consistency property has been called later (not by Nash himself) independence of irrelevant alternatives, causing much confusion." While Nash's axiom is different from the version of IIA used in Arrow's famous \emph{impossibility theorem}, Arrow himself was not always careful to distinguish between the two.  The history of these axioms (and of their confusion) is more complicated than is often understood; anyone interested should consult Salles's slide talk \cite{Salles18} on the topic. }$^{,}$\footnote{
Nash's axioms fall into two groups: the two \textit{consequential} axioms of scale invariance and Nash IIA distinguish between Nash's solution and other  solutions that have been proposed for the bargaining problem (such as that of Kalai and Smorodinsky \cite{KalSmo75}, or the egalitarian solution of Kalai \cite{Kal77}), while the \textit{innocuous} axioms constitute a \textit{sine qua non} shared by all reasonable solutions.  Thus our axioms for $h'$ include analogues of both of Nash's consequential axioms.}
 This suggests that the parallels between Nash's bargaining solution and the $h'$ index go beyond the shared element of scale invariance.


Our starting motivation for introducing scale invariance was that it promised to correct distortions that arise in rankings induced by any index that rests on an an implicit assumption of scale comparability.
But our investigations revealed additional advantages for the scale invariant versions, showing them to be more dynamic, growing over time via increments that are smaller and more frequent, with the enhanced resolution yielding fewer ties. Moreover, simulations that add Poisson noise to Hirsch's simple deterministic model show that scale invariance does not increase the number of ``accidental" reversals, wherein random irregularities cause researcher $A$ to receive a lower index value than $B$, even though $A$'s  productivity and impact are both slightly higher than $B$'s. We take this to be evidence that the increased resolution of the scale invariant versions is not false precision.  Moreover, scale invariance yields a better balance between contributions from publications with the highest visibility and publications having somewhat less impact.  In the case of $h'$, this means that the max-area inscribed rectangle tends to be located more centrally within the $\mathcal B({\bf x})$ region, compared with the Hirsch square, producing a closer balance between the area lost to the vertical tail and that lost to the horizontal tail.

%
%
%
%
%
The article is organized as follows: After the Section \ref{SEC:INTRO}.1 discussion of related literature, including the $\chi$-index of \cite{FHLB}, Section \ref{SEC:INTRO}.2 considers some of the trade-offs inherent in choosing a citation index that rewards balance between impact and productivity.  Section \ref{SEC:DEFS} then reviews three ``standard" scientific citation indices---those introduced by Hirsch, Woeginger, and Egghe---and proposes a fourth.  We introduce scale-invariant citation indices in Section \ref{SEC:NEW-IND} and provide a list of axioms for such indices. Section \ref{SEC:ScaleInvariantJustification} discusses the idea behind scale invariance, compares standard indices of Hirsch and Woeginger with their scale-invariant versions, and argues the need for indices to be scale-invariant.  The argument is buttressed by parallels with the two-person bargaining problem considered by Nash. Section \ref{SEC:LinearGrowth} considers the simple deterministic case and shows that the Hirsch and Woeginger indices, as well as their scale invariant versions, grow linearly as a function of time, and Section \ref{SEC:Axiomatization} provides two similar axiomatizations of the scale-invariant Hirsch index $h'$, one of which uses Linear Growth as an axiom. Advantages of $h'$ over the original version $h$ of the Hirsch index include improved robustness and resolution when noise is added to the standard deterministic model, evidence for which is presented by the simulations discussed in Section \ref{SEC:Simulations}.
 Section \ref{SEC:Goodness} highlights what we see as the four principal advantages of $h'$ over the original version $h$ of the Hirsch index. Our concluding Section \ref{SEC:MoreQuestions} summarizes the goals of the paper, and briefly describes some promising avenues for further research.

\subsection{Related work} \label{SEC:DEFS}

The $\chi$ index of Fenner et al. \cite{FHLB} is identical to the index $h'$
we discuss here.  Their paper explicitly discusses the ordinal version of
scale invariance (the one we call here SInv) rather than the cardinal version SSInv (see Section
\ref{SEC:NEW-IND}.1).\footnote{Their definition considers only vertical scalings that
multiply the number of citations for each paper by some constant factor,
and not the ``horizontal" ones that clone each paper.  Note, however, that in the
presence of the symmetry axiom of Section \ref{SEC:NEW-IND}.1, either form of scale invariance
implies the other.}  The principal focus is on an analysis of real citation
records: how do the values assigned by Hirsch's index $h$ differ from those
given by its scale-invariant version $\chi = h'$?

Two data sets are analyzed.  The first consists of approximately 90,000 citation records of scholars from diverse fields, each of whom had ``validated their Google Scholar accounts."\footnote{Presumably the validation consisted of adding any missing papers, and merging duplicate listings of the same paper, both of which are potential sources of error in raw Google Scholar records. Worse is that the matter of what constitutes a duplicate is up for interpretation for a paper that may have both conference versions and a journal version, which are not identical but may have substantial overlap.} The second contains the records of 99 Nobel laureates.  The authors find a high correlation between $h$ and $h'$, using either Spearman or Pearson correlation coefficients. In other words, $h$ and $h'$ are generally consistent.  Nonetheless ``there are a substantial number of profiles for which $\chi$ is significantly larger than $h$."

Additionally, authors are grouped as being ``influential," ``prolific," or ``balanced," corresponding respectively to maximal area rectangles that are significantly higher than they are wide ($c > p,$ for the simple deterministic model), wider than they are high ($p > c$), or roughly square ($c \approx p$). This determination was made utilizing a bootstrapping (sampling with replacement) approach to resampling author citation vectors. Most authors are ``influential": $53\%$ of the Google Scholar group, and $80\%$ of the Nobelists.

Fenner et al. interpret their data as evidence that, ``for very
influential
researchers, such as Nobel laureates, when $\chi > h,$ the $h$-index undervalues their contribution." They conclude that ``the $\chi$-index is beneficial and could lead to a more satisfactory ranking of researchers than that obtained using the $h$-index."


An axiomatic characterization of the $\chi$ index is given in the follow-up article \cite{LFB2019}, by three of the four authors of the first paper.  The axioms employed, however, are different in character from those we use in Section \ref{SEC:NEW-IND}.1. Scale invariance itself is not among the three axioms used.  Monotonicity is one of them, but the flavor of the other two axioms is more algebraic than normative. These two axioms consider citation records having a specific form, consisting of $p$ papers, each of which has the same number $c$ of citations (equivalently, the bar graph forms a perfect rectangle).  One axiom sets the value of an index on any such record to be $\sqrt{pc}$, while the the other requires the value on an arbitrary citation record $\bf{x}$ to equal its value on some record $\bf{y} \leq \bf{x}$ having this rectangular form. 

The other work most directly relevant to what we do here is \cite{Woe08a}, in which Woeginger proposes his alternative $w$ to Hirsch's index, based on inscribing an isosceles triangle, rather than a square, in $\mathcal B({\bf x})$. The paper offers axiomatic characterizations of both $h$ and $w$. 

\subsection{The cost of balance}\label{BalanceCost}

Choosing a citation index that rewards balance has a cost, in that balance is inconsistent with other properties that may be seen as desirable (or even, by some authors, seen as \emph{sine qua nons}). We'll mention two of these.  The \emph{independence} axiom requires, for every pair of researchers and every positive integer $k$, that if each researcher adds a single new paper with $k$ citations, then the index value for Researcher $1$ is at least as high as that for Researcher $2$ after the addition if and only if it had been at least as high before.\footnote{One direction of this iff implies that adding the new paper never breaks a tie (between two previously tied researchers); the other direction implies that the addition never creates a tie (between two who had not been tied).  Additionally, either direction alone implies that the addition never reverses which of the two has a strictly higher index value. Thus, independence demands that this particular type of common addition to the record does not change, in any of these senses, the ordinal ranking of the two.}
See \cite{Marchant2009a},  \cite{Marchant2009b}, \cite{WE1} (where the property is called \emph{consistency}), or \cite{BM1}.  
The second property, which we will call \emph{batching consistency} here (it is referred to as \emph{consistency} in \cite{BM2}) considers the use of a citation index to compare two equal sized groups of researchers, such as two different Economics departments, by treating each group as if it were a single individual who authored all papers produced by members of the group.  Loosely speaking, batching consistency then demands of any index $g$ that if we can pair off members of Department $1$ with those of Department $2$, in such a way that each individual member of Department $1$ is rated at least as highly by $g$ as is the Department $2$ member paired to her, then $g$ rates the first group at least as highly as the second. 

The Hirsch index fails to satisfy either of these axioms (\cite{WE2}, \cite{BM1}, \cite{WE3}).  But should this come as a surprise?  Both properties have a strong whiff of linearity, requiring that the whole be seen as the sum of its parts.  Balance is quite another thing, depending on the relationship of those parts (the heights of the bars in the bar graph representation) to one another. For a citation index based on balance, it seems entirely understandable that with a citation record imbalanced towards impact, adding a new bar of height $k$ might increase the index value by enhancing productivity, while adding that same bar to a different citation record---one with high productivity but lower impact---would yield no change in the index. 

A very similar analysis applies to batching consistency. Suppose the citation records of the individual researchers in Department $1$ are all similarly imbalanced towards impact, with each researcher having produced a small numbers of highly cited papers. Then combining them in the manner of batching consistency enhances productivity (there are more papers), bringing it into better balance with impact. An index based on balance may see quite a boost when applied to the combined record.  In Department 2, the citation records of the individual researchers might instead all be imbalanced toward productivity.  Notice that in this case, combining the records does not enhance impact (papers from the individual records gain no additional citations when they are viewed as part of the combined record), so that for an index based on balance, there would be no corresponding boost for Department 2.  This story easily translates into a failure of batching consistency.  

That the Hirsch index fails to satisfy these two axioms should be seen as a predictable and natural conflict between properties, not as the sudden revelation of a previously hidden defect.  Such conflicting desiderata are quite common in the mathematical social sciences, where they may be treated as \emph{trade-offs}; we can't have both, so we must choose the one we see as more important. In the study of axiomatic properties of voting rules, for example, one possible approach is to weight the relative importance of conflicting properties according to the context of the election.

\section{Scientific citation indices
} \label{SEC:DEFS}

Recalling the definition of citation record from Section \ref{SEC:INTRO}, we let $X$ denote the set of all possible such records ${\bf x} = (x_1,x_2,\dots,x_{l({\bf x})})$. 
Elements ${\bf x} \in X$ are denoted in boldface; by lightface $x \in X$ we mean that ${\bf x} = (x)$, so that ${\bf x}$ is formed by a single paper with $x>0$ citations. We can easily visualize ${\bf x}$ by means of a \emph{step-function} representation $s_{{\bf x}}$ on the interval $[0,l({\bf x})]$, a \emph{correspondence} $c_{{\bf x}}$ in the same domain, or a \emph{bar-graph subset} $\mathcal B({\bf x})$ of the first quadrant of ${\mathbb{R}}^2$. The respective definitions for a given vector ${\bf x} \in X$ are as follows. Let
$$
    s_{\bf x} (x) = \left\{
                      \begin{array}{ll}
                        x_1, & \hbox{if} \ \ x=0 \\
                        x_i, & \hbox{if} \ \ x \in (i-1,i]  \; \text{and} \; 1 \leq i \leq l({\bf x}) \\
                      \end{array}
                    \right.
$$
Note that $s_{{\bf x}}$ is the minimal nonnegative monotonically decreasing function on $[0,l({\bf x})]$ that agrees with ${\bf x}$ on positive integers. It has discontinuities  on the right at integers $i$ such that $0<i<l({\bf x})$ and $x_i > x_{i+1}$.

The correspondence $c_{{\bf x}}$ can be obtained form $s_{{\bf x}}$ by just adding vertical segments at these discontinuities:
$$
    c_{\bf x} (x) = \left\{
                      \begin{array}{ll}
                        x_1, & \hbox{if} \ \ x=0 \\
                        x_i, & \hbox{if} \ \ x \in (i-1,i)  \; \text{and} \; i \leq l({\bf x}) \\
                        \left[ x_{i+1},x_i \right], & \hbox{if} \ \ x=i, \; \text{and} \; 1 \leq i \leq l({\bf x}) \\
                      \end{array}
                    \right.
$$
with $x_{l+1} = 0$ and $\{ a \} = [a,a]$ identified with $a$.
The compact bar-graph subset $\mathcal B({\bf x}) \subseteq {\mathbb{R}}^2$ of non-negative real components is limited above  by $s_{{\bf x}}$ so that
$(a,b) \in \mathcal B({\bf x})$ \emph{if and only if} $0 \leq a \leq l({\bf x})$ and $0 \leq b \leq s_{{\bf x}}(a)$. 

Conversely, let $s$ be a decreasing step-function on a given interval $[0,l]$, with $l \in \mathbb{N}$, with image set contained in $\mathbb{N}$ and with possible discontinuities only on the right and at integers in $[1,l]$.
The restriction to integers of this step-function defines a citation record ${\bf x}_s$ such that $x_i= s(i)$ for all $i=1,2,\dots, l$. From $s$ one can generate $c$ and $\mathcal B$ and from any of them ${\bf x}$ is easily recovered.\footnote{For instance, the $i$-th coordinate of ${\bf x}_{\mathcal B}$ is obtained from $\mathcal B$ as $x_i = \max \{y \! :  (i,y) \in \mathcal B \}$.}

\begin{example}  \label{curves}
Let ${\bf x} = (11,7,6,6,6,4,4,4,3,3,2,2,1,1,1)$ be the research record of a scientist $R$, so that $l({\bf x})=15$. Figure~\ref{fig:curves} shows $s_{\bf x} (x)$,  $c_{\bf x} (x)$ and $\mathcal B({\bf x})$.
\end{example}

\begin{figure}[H]
\centering
\subfigure[The $s_{\bf x} (x)$ function.]{\begin{tikzpicture}[scale=0.55]
\begin{axis}[
unit vector ratio*= 1 1 1,
xlabel={$i$},ylabel={\rotatebox{270}{$x_i$}},ymin=0,ymax=15,xmin=0,xmax=20]
\coordinate (Point) at (0,11);
\draw [black,fill] (Point) circle (1.5pt);
\coordinate (Point) at (1,11);
\draw [black,fill] (Point) circle (1.5pt);
\coordinate (Point) at (2,7);
\draw [black,fill] (Point) circle (1.5pt);
\coordinate (Point) at (5,6);
\draw [black,fill] (Point) circle (1.5pt);
\coordinate (Point) at (8,4);
\draw [black,fill] (Point) circle (1.5pt);
\coordinate (Point) at (10,3);
\draw [black,fill] (Point) circle (1.5pt);
\coordinate (Point) at (12,2);
\draw [black,fill] (Point) circle (1.5pt);
\coordinate (Point) at (15,1);
\draw [black,fill] (Point) circle (1.5pt);
\addplot
[const plot,thick,draw=black]
coordinates
{(0,11) (1,11)};
\addplot
[const plot,thick,draw=black]
coordinates
{(1,7) (2,7)};
\addplot
[const plot,thick,draw=black]
coordinates
{(2,6) (5,6)};
\addplot
[const plot,thick,draw=black]
coordinates
{(5,4) (8,4)};
\addplot
[const plot,thick,draw=black]
coordinates
{(8,3) (10,3)};
\addplot
[const plot,thick,draw=black]
coordinates
{(10,2) (12,2)};
\addplot
[const plot,thick,draw=black]
coordinates
{(12,1) (15,1)};
\end{axis}
\end{tikzpicture}} \hspace{10mm}
\subfigure[The $c_{\bf x} (x)$ correspondence.]{\begin{tikzpicture}[scale=0.55]
\begin{axis}[
unit vector ratio*= 1 1 1,
xlabel={$i$},ylabel={\rotatebox{270}{$x_i$}}, ymin=0,ymax=15,xmin=0,xmax=20,
area style]
\addplot
[const plot,thick,draw=black]
coordinates
{(0,11) (1,11) (1,7)};
\addplot
[const plot,thick,draw=black]
coordinates
{(1,7) (2,7) (2,6)};
\addplot
[const plot,thick,draw=black]
coordinates
{(2,6) (5,6) (5,4)};
\addplot
[const plot,thick,draw=black]
coordinates
{(5,4) (8,4) (8,3)};
\addplot
[const plot,thick,draw=black]
coordinates
{(8,3) (10,3) (10,2)};
\addplot
[const plot,thick,draw=black]
coordinates
{(10,2) (12,2) (12,1)};
\addplot
[const plot,thick,draw=black]
coordinates
{(12,1) (15,1) (15,0)};
\end{axis}
\end{tikzpicture}} \hspace{10mm}
\subfigure[The $\mathcal B({\bf x})$ bar-graph subset.]{\begin{tikzpicture}[scale=0.55]
\begin{axis}[
unit vector ratio*= 1 1 1,
xlabel={$i$},ylabel={\rotatebox{270}{$x_i$}},ymin=0,ymax=15,xmin=0,xmax=20,
area style]
\addplot
[const plot,fill=cyan,draw=black]
coordinates
{(0,11) (1,7) (2,6) (3,6)
(4,6) (5,4) (6,4) (7,4)
(8,3) (9,3) (10,2) (11,2) (12,1) (13,1) (14,1) (15,0)}
\closedcycle;
\end{axis}
\end{tikzpicture}}
\caption{Three ways to visualize the citation record ${\bf x}$.} \label{fig:curves}
\end{figure}
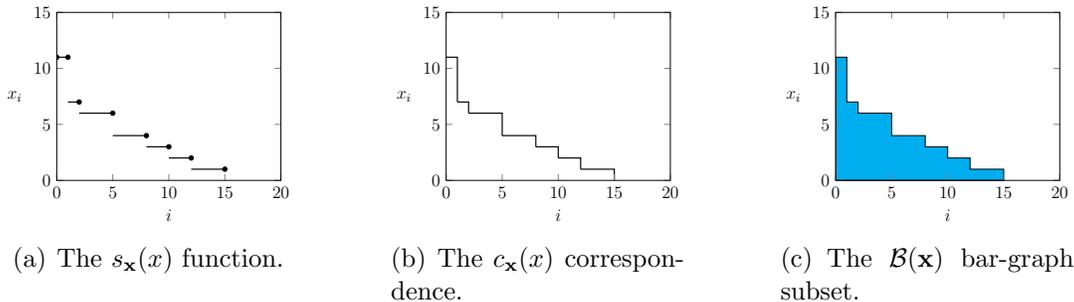

\begin{definition} \label{D:sci}
A \emph{scientific citation index} (or index, for short) is a function $g$ from the set $X$ of citation records into the interval $[1,+\infty)$ of real numbers,  such that $g(1)=1$.
\end{definition}

While some authors require an index to take only integer values, note that Definition \ref{D:sci} is free of this restriction, thus allowing an index to be more dynamic and break ties more frequently. The condition $g(1)=1$ should be interpreted as a normalization factor that avoids indices that only differ by a multiplicative factor. All integer indices  we know satisfy this condition, which can be seen as a \emph{conditio sine qua non} for a citation index. A second \emph{conditio sine qua non} is that of monotonicity, discussed in Section~\ref{SEC:NEW-IND}, but we list it as an axiom, rather than incorporating it in  Definition~\ref{D:sci}.

The rest of this section is devoted to recalling precise definitions for the Hirsch and Woeginger indices.

\begin{definition} \label{D:h}
The $h$-index (Hirsch-index) is the scientific citation index $h\!:X \rightarrow [1,+\infty)
$ assigning to each vector ${\bf x}$ the value $h({\bf x}) := \max\limits_{1 \leq k \leq l({\bf x})} \{ k \! :  x_k \geq k \} $.
\end{definition}
An equivalent definition is $ h({\bf x})  :=  \max\limits_{1 \leq k \leq l({\bf x})} \min \{ k, x_k \}$.

The Hirsch index coincides with the side length of the ``Hirsch square"---the largest square that lies below the step-function $s_{\bf x}$, is contained in the first quadrant of $\mathbb{R}^2$, and has one corner at the origin (whence the diametrically opposite corner is the intersection point of the curve $c_{\bf x}$ with the line $y=x$). Equivalently, the value of the $h$-index is the square root of the area of the largest square contained in $\mathcal B({\bf x})$, among squares containing the origin.
Note that all points on the curve $c_{\bf x}$ have at least one integer coordinate. Thus, this intersection point with the line $y=x$  is a point on  $c_{\bf x}$ with two equal integer coordinates. This means that if we substituted $\mathbb{N}$ for $[1,+\infty)$ as $h$'s image in Definition~\ref{D:h} (which is how Hirsch~\cite{Hir05} originally defined his index) the index would be unchanged.

\begin{definition} \label{D:w}
The $w$-index (Woeginger-index) is the scientific citation index $w\!:X
\rightarrow [1,+\infty)
$ assigning to each vector ${\bf x}$ the value $w({\bf x}) := \max\limits_{1 \leq k \leq l({\bf x})} \{ k \! :  x_m \geq k-m+1 \ \text{for all} \ m \leq k \} $.
\end{definition}

The Woeginger index coincides with the leg length of the largest isosceles right triangle that lies below the step-function $s_{\bf x}$, is contained in the first quadrant of $\mathbb{R}^2$, and whose right angle is at the origin (whence one leg lies along the horizontal axis, and another lies along the vertical axis).  Equivalently, the value of the $w$-index is the square root of twice the area of the largest isosceles right triangle contained in $\mathcal B({\bf x})$, among triangles whose right angle is at the origin.
Note that the largest isosceles right triangle contained in $\mathcal B({\bf x})$ has hypotenuse lying along a line with equation $y=-x+w$ for some $0<w \leq l({\bf x})$;
 this line intersects the curve $c_{\bf x}$ in at least one point with integer coordinates, whence $w \in \mathbb{N}$.  This means that the index would be unchanged if we substituted $\mathbb{N}$ for $[1,+\infty)$ as $w$'s image in Definition~\ref{D:h} (which is how  it was originally defined by Woeginger~\cite{Woe08a}).

Loosely speaking, the $h$-index maximizes the area of a scaled copy of the ${\ell}_{\infty}$ unit ball centered at the origin, having an intersection with the first quadrant that is contained in ${\mathcal B}({\bf x})$, while the
$w$-index does the same for a scaled copy of the ${\ell}_{1}$ unit ball. Other indices can be defined by using ${\ell}_{p}$ balls for other values of $p$. For example, for the case $p=2$ (corresponding to the Euclidean norm), we would be maximizing the area of a scaled copy of a true unit disc:

\begin{definition} \label{D:c}
The $c$-index is the scientific citation index $c:X \rightarrow 
[1,+\infty)
$
that assigns to vector ${\bf x}$ the value $c({\bf x}) := \max \left\{ k \, : \, k \leq \sqrt{(i-1)^2+(x_i)^2}  \ \text{for all} \ i=1, \dots, l({\bf x}), \, l({\bf x})+1 \right\}$, where $x_{l({\bf x})+1}=0$.
\end{definition}
As far as we know this is a new scientific citation index. It can also be defined as
$$
c({\bf x}) := \min\limits_{1 \leq k \leq l({\bf x})+1} \left\{
((k-1)^2+(x_k)^2)^{1/2}
\right\}
$$
where $x_{l({\bf x})+1}=0$.
Note that the $x_{l({\bf x})+1}+1$ points $(i-1,x_i) \in c_{\bf x}$ have integer coordinates. Thus, the $c$-index is the square root of an integer.

\medskip

Some other indices, not based on scaled copies of unit balls, have attracted interest.
One prominent example in this regard is the Egghe index, see \cite{Egg06, AdKo15}.

\begin{definition} \label{D:EggheX}
The $e$-index is the scientific citation index $e:X \rightarrow \mathbb{N}$ that assigns to vector ${\bf x}$ the value
$
e({\bf x}) := \max \left\{ \  k \in \mathbb{N} \, : \,   \sum\limits_{i=1}^{\min\{k, l({\bf x}\}\}} x_i \geq k^2  \ \right\}$.
\end{definition}


An alternate but equivalent definition is convenient for our purposes.  For citation records ${\bf x}$ and ${\bf x'}$ we'll say that ${\bf x}$ \emph{cumulatively dominates} ${\bf x'}$  if $l({\bf x}) \leq l({\bf x'})$ and

\begin{equation}\label{CumDomEqn}
\sum\limits_{i=1}^{\min\{j,\, l({\bf x})\}} x_i \; \; \, \geq  \sum\limits_{i=1}^{\min\{j,l({\bf x'})\}} x'_i
\end{equation}

\noindent holds for each $j \in \mathbb{N}$; informally, ${\bf x'}$ arises from ${\bf x}$ by shifting citations from more cited papers to less cited papers (potentially including ``new" papers with indices greater than $l({\bf x})$), smooshing parts of the vertical tail of
 $\mathcal B({\bf x})$ to the right.  The Egghe index $e({\bf x})$ can now equivalently be defined as follows: find the maximum value,  taken over all citation records ${\bf x'}$ cumulatively dominated by ${\bf x}$, of the Hirsch index $h({\bf x'})$.  Loosely speaking, the Egghe index allows citations from the vertical tail to contribute to the dimensions of the Hirsch square, but continues to discount contributions from the horizontal tail.  From either definition, one can see that the $e$-index can assign a value greater than the number of publications by a researcher.

Note that for the alternate definition given here, it does not matter whether we consider $e$ to be a map from $X$ to $\mathbb{N}$ or a map from $X$ to $[1, + \infty )$; as with the Hirsch and Woeginger indices, the outputs will be integers whether or not we explicitly require them to be so.  The story is different for the original definition (which is why the output variable $k$ is explicitly required to belong to $\mathbb{N}$ in Definition \ref{D:EggheX} above).  The most straightforward variant of \ref{D:EggheX} that opens the door to non-integer outputs would seem to be as follows:
$\overline{e}({\bf x}) := \max \left\{ \  k \in [1, + \infty ) \, : \,   \sum\limits_{i=1}^{\min\{ \lfloor k \rfloor,\, l({\bf x})\}} x_i \geq k^2  \ \right\}$.
Example \ref{EXAMPLE} shows that $\overline{e}({\bf x}) > e({\bf x}) $ sometimes holds.

A second virtue of the alternate definition is that it allows one to pose a scale-invariant version of the Egghe index.\footnote{
It is not clear to us how one might introduce scale invariance into the original Definition \ref{D:EggheX}.}
As we see in the next section, however, that version is fundamentally flawed, suggesting that there may not exist any reasonable scale-invariant version of Egghe's index.

\begin{example}  \label{EXAMPLE}  (Example \ref{curves} revisited)
Let ${\bf x} = (11,7,6,6,6,4,4,4,3,3,2,2,1,1,1)$ be the research record of a scientist $R$, so that $l({\bf x})=15$ and for $R$ we have: $h=5$, $w=8$, $c=6$, $e=6$ and $\overline{e}=\sqrt{40}$. Figure~\ref{fig:BasicIndices} shows the respective regions that give rise to the first three of these numbers.
\end{example}

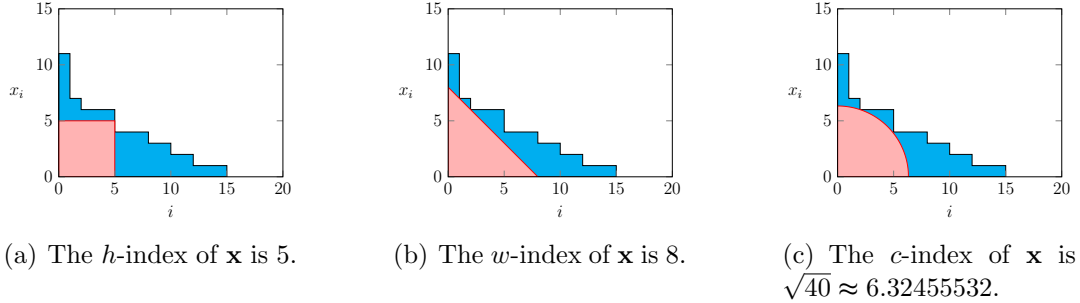
\begin{figure}[H]
\centering
\subfigure[The $h$-index of ${\bf x}$ is $5$.]{\begin{tikzpicture}[scale=0.55]
\begin{axis}[
unit vector ratio*= 1 1 1,
xlabel={$i$},ylabel={\rotatebox{270}{$x_i$}},ymin=0,ymax=15,xmin=0,xmax=20,
area style]
\addplot
[const plot,fill=cyan,draw=black]
coordinates
{(0,11) (1,7) (2,6) (3,6)
(4,6) (5,4) (6,4) (7,4)
(8,3) (9,3) (10,2) (11,2) (12,1) (13,1) (14,1) (15,0)}
\closedcycle;
\addplot coordinates
{(0,5) (5,5) (5,0)} \closedcycle;
\end{axis}
\end{tikzpicture}} \hspace{10mm}
\subfigure[The $w$-index of ${\bf x}$ is $8$.]{\begin{tikzpicture}[scale=0.55]
\begin{axis}[
unit vector ratio*= 1 1 1,
xlabel={$i$},ylabel={\rotatebox{270}{$x_i$}}, ymin=0,ymax=15,xmin=0,xmax=20,
area style]
\addplot
[const plot,fill=cyan,draw=black]
coordinates
{(0,11) (1,7) (2,6) (3,6)
(4,6) (5,4) (6,4) (7,4)
(8,3) (9,3) (10,2) (11,2) (12,1) (13,1) (14,1) (15,0)}
\closedcycle;
\addplot coordinates
{(0,8) (8,0)} \closedcycle;
\end{axis}
\end{tikzpicture}} \hspace{10mm}
\subfigure[The $c$-index of ${\bf x}$ is $\sqrt{40} \thickapprox 6.32455532$.]{\begin{tikzpicture}[scale=0.55]
\begin{axis}[
unit vector ratio*= 1 1 1,
xlabel={$i$},ylabel={\rotatebox{270}{$x_i$}},ymin=0,ymax=15,xmin=0,xmax=20,
area style]
\addplot
[const plot,fill=cyan,draw=black]
coordinates
{(0,11) (1,7) (2,6) (3,6)
(4,6) (5,4) (6,4) (7,4)
(8,3) (9,3) (10,2) (11,2) (12,1) (13,1) (14,1) (15,0)}
\closedcycle;
\pgfplotsset{samples=500}
\addplot+[domain=0:6.33] {sqrt(40-x^2)} \closedcycle;
\end{axis}
\end{tikzpicture}}
\caption{Indices inspired by $\ell_{\infty}$, $\ell_{1}$ and $\ell_{2}$  metrics respectively.} \label{fig:BasicIndices}
\end{figure}


\section {Scale-invariant scientific impact indices and some axioms}  \label{SEC:NEW-IND}

We start by providing informal, geometric definitions of our proposed scale-invariant versions $h'$, $w'$ and $c'$  of the $h$, $w$ and $c$ indices, with  more precisely phrased definitions to follow.
The value $h'({\bf x})$ of the scale-invariant Hirsch index is given by $\sqrt{ab}$,  where $a \times b$ are the dimensions of any rectangle of greatest area, among all rectangles contained in $\mathcal B({\bf x})$ that have one corner at the origin.  It follows that the diametrically opposite corner is the intersection point of the line through the origin having equation $y=(b/a)x$ with the step-function $s_{\bf x}$. This intersection occurs in a point of  discontinuity on the right for $s_{\bf x}$, so that the rectangle contacts $c_{\bf x}$ at this point as well as at least two other points that also have integer coordinates. Note that $\sqrt{ab}$ is the square root of that maximal area, and that there may exist more than one rectangle of maximal area.

The value $w'({\bf x})$ of the scale-invariant Woeginger index is given by $\sqrt{ab}$,  where $a$ and $b$ are the leg lengths of any right triangle of greatest area, among all right triangles that are contained in $\mathcal B({\bf x})$ with their right angle at the origin.
It follows that the hypotenuse touches the curve $c_{\bf x}$ in at least two points with integer coordinates.\footnote{
A short argument shows that a hypotenuse that touches $c_{\bf x}$ in only one point can be rotated through a small angle about that point, in such a way that the triangle gains area. This argument does not apply, of course, to the isosceles triangle of the original $w$-index, which may contact the $c_{\bf x}$ graph at a unique point $(i,x_i)$, corresponding to a particular publication; the index can then increase in value only if a new citation increases $x_i$ in particular.  This distinction may help explain why the scale-invariant versions respond more frequently, and flexibly, to an evolving publication record.
}
 Note that $\sqrt{ab}$ is the square root of double that maximal area, and that there may exist more than one triangle of maximal area.

The value $c'({\bf x})$ of the scale-invariant $c$-index is given by $\sqrt{ab}$,  where $a$ and $b$ are the semi-major and the semi-minor axes of any elipse of greatest area, among all ellipses whose intersection with the first quadrant of $\mathbb{R}^2$ is contained in $\mathcal B({\bf x})$, whose center is at the origin and whose semi-major and the semi-minor axes lie along the coordinate axes. The definition implies that the ellipse touches the curve $c_{\bf x}$ in at least two points with integer coordinates.  More generally:

\begin{definition}
\label{ShapeIndex}
Let $\mathcal{S}\mathcal{Q}_1 = [0,1] \times [0,1]$ denote the unit square, and $\mathcal{L}_1 = {\big (}[0,1] \times \{ 0 \}{\big )} \cup  {\big (}\{ 0 \} \times [0,1]{\big )}$ contain the points along the lower and left boundaries of $\mathcal{S}\mathcal{Q}_1$. Choose any closed convex region $S$ with $\mathcal{L}_1 \subseteq S \subseteq \mathcal{S}\mathcal{Q}_1$ that is symmetric about the line $y=x$.\footnote{Without the convexity requirement, Definition \ref{ShapeIndex} would seem to allow indices that fail to impose the balance between productivity and impact that Hirsch had in mind (because non-convex regions, such as  $S = \mathcal{L}_1$, might fail to lop off the tails of $\mathcal B({\bf x})$).  Note that the convexity requirement implies that the smallest admissible shape $S$ leads to the scale-invariant version $w'$ of Woeginger's index, while the largest such $S$ leads to $h'$.
} Then $S$ serves as the basis of a \emph{scale-invariant symmetric shape citation index} defined by $g'_S({\bf x}) = \sqrt{ab}$; where $a,b$ are positive real numbers chosen to maximize $\sqrt{ab}$ subject to the requirement that stretching $S$ horizontally by factor of $a$ and vertically by a factor of $b$ yields a region $aS^{\leftrightarrow b}$ contained as a subset of $\mathcal B({\bf x})$.
\end{definition}

Any such region $S$ satisfies $\mathcal{T}_1 \subseteq S \subseteq \mathcal{S}\mathcal{Q}_1$, where $\mathcal{T}_1$ denotes the isoceles right triangle obtained as the convex hull of $ \mathcal{L}_1$, and it follows that $h'({\bf x}) \leq g'_S({\bf x}) \leq w'({\bf x})$ holds for all ${\bf x}$; thus $w'$ assigns the largest values among scale-invariant symmetric shape citation indices, and $h'$ assigns the smallest.   We don't argue, here, that any of these other scale-invariant indices offers  specific advantages over $h'$ or $w'$.  They do, however, suggest the broad variety of alternative scale-invariant indices, and it seems possible that with the ``right" choice of shape $S$,  $g'_S$ satisfies some axiom (alternative to the \emph{Max-Bounded} axiom  characterizing $h'$) having independent normative appeal.

A formal definition of the scale-invariant Hirsch index is straightforward. While the other two are a little more elaborate, they can easily be obtained as the maxima of respective optimization problems.
\begin{definition} \label{D:sih}
The scale-invariant Hirsch index is the scientific citation index $h'\! :X \rightarrow [1,+\infty)$ given by
$$
h'({\bf x}) := \max\limits_{1 \leq k \leq l({\bf x})} \, \, \left\{ (k \cdot x_k)^{\frac 12} \right\}.
$$
\end{definition}

\noindent This index can be seen as a member of the following parameterized family of \emph{scale invariant Hirsch powers}: for each $a>0$, let
\begin{equation}  \label{EQ:HirschParametric}
h_a({\bf x}) :=  \max\limits_{1 \leq k \leq l({\bf x})} \, \, \left\{ (k \cdot x_k)^{a} \right\},
\end{equation}
with $a = 1/2$ leading to $h'$ in particular, so that  $h'$ and $h_{1/2}$  denote the same index.  Our initial set of axioms, in Section \ref{SomeAxioms},  characterize the entire family of scale invariant Hirsch powers. One additional axiom, \emph{Linear Growth},\footnote{
Discussed briefly in the introduction, and stated precisely in Section \ref{SomeAxioms}.
} selects $h'$ in particular.

Comparison to the Nash Bargaining Theorem is facilitated by noting the following equivalent formulation of the $h'$-index:
$$
h'({\bf x}) = \max_{(x,y) \in \mathcal B({\bf x})} (xy)^{1/2},
$$
while $h_a$
is given by
$$
h_a({\bf x}) = \max_{(x,y) \in \mathcal B({\bf x})} (xy)^{a}.
$$
Note that for all $a>0$ we have:
\begin{equation} \label{EQ:argmax-sih-a}
\arg \max_{(x,y) \in \mathcal B({\bf x})} (xy)^{a} = \arg \max_{(x,y) \in \mathcal B({\bf x})} xy,
\end{equation}
so that the points in  $\mathcal B({\bf x})$ for which the product $xy$ is maximized do not depend on the power $a$.
Thus:
\begin{itemize}

\item for any $(x,y) \in \arg \max\limits_{(x,y) \in \mathcal B({\bf x})} xy$ and $a>0$, the index $h_a$ is given by $(xy)^a$,
\item $h_a({\bf x}) = [h_1({\bf x})]^a$ for each citation record ${\bf x}$,
\item the ranking of researchers according to index value is independent of $a$; for any two publication records ${\bf x},{\bf y}$ and two real numbers $a, b > 0$, $h_a({\bf x}) > h_a({\bf y}) \Leftrightarrow h_b({\bf x}) > h_b({\bf y})$.

\end{itemize}

\medskip

The other two indices, which play a less prominent role in this article, can be  be obtained by solving the following respective optimization programs.

\begin{definition} \label{D:siw}
The scale-invariant Woeginger index is the scientific citation index \\ $w'\! :X \rightarrow [1,+\infty)$ that assigns to vector ${\bf x}$ the optimal value of the following problem:
$$
\begin{array}{cl}
  \max  &  (c\cdot d)^{\frac 12} \\
     &    \\
  s.t. &  -\frac dc k + d \leq x_{k+1} \ \  \text{for all} \ \ k=0,1,2, \dots, l({\bf x}) \\
   & c \geq 0, \ d \geq 0
\end{array}
$$
where $x_{l({\bf x})+1} = 0$.
\end{definition}

The above constraints demand that the hypotenuse (with equation $y = -\frac dc x +d$ and $0 \leq x \leq c$) of the right triangle lie weakly under $c_{\bf x}$; any optimal solution is derived from values $c$ and $d$ such that this hypothenuse 
touches $c_{\bf x}$ in at least two points with integer coordinates.
Thus, at least two of the first $l({\bf x})+1$ constraints are verified with equality for the optimal hypotenuse.
It would also be possible to analogously define the family of  \emph{scale invariant Woeginger powers}, $w_a$, verifying
analogous properties to those in Equation \eqref{EQ:argmax-sih-a} and in the three subsequent bullet points.

An optimization program for finding the $c'$-index is the following.
\begin{definition} \label{D:siellipse}
$$
\begin{array}{cl}
  \max  &  (c \cdot d)^{\frac 12} \\
     &    \\
  s.t. &  \frac dc \sqrt{c^2-k^2} \leq x_{k+1} \ \  \text{for all} \ \ k=0,1,2, \dots, l({\bf x}) \\
   & c \geq 0, \ d \geq 0
\end{array}
$$
\end{definition}
The constraints are deduced from ellipses instead of lines and the optimal ellipse touches  $c_{\bf x}$ in at least two points with integer coordinates.

Let's return to Example~\ref{EXAMPLE} and compute the scale-invariant versions of the three indices considered.  See the respective pictures and optimal solutions in Figure \ref{fig:SIIndices}. 
Observe that the optimal ellipse obtained in computing $c'$ passes through $(2,6)$ and $(5,4)$. From the general ellipse equation   $b^2x^2+a^2y^2=a^2b^2$, with semi-axes $a>0$ and $b>0$, we deduce that $a$ is the positive root of $5z^2-209=0$ and $b$ is two times the positive root of $21z^2-209=0$.

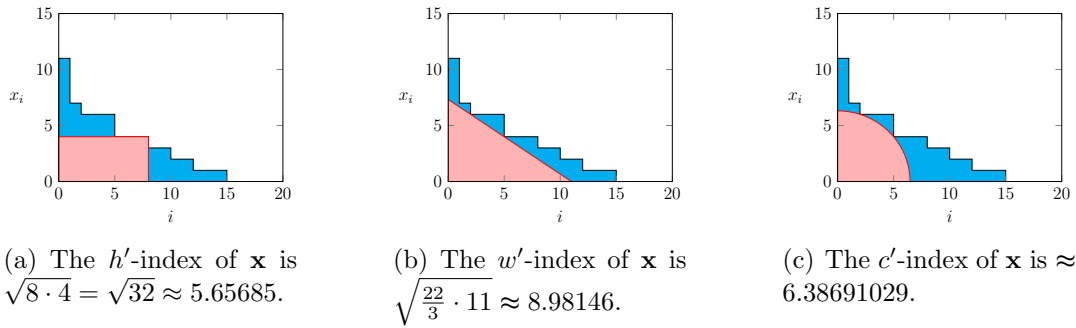
\begin{figure}[H]
\centering
\subfigure[The $h'$-index of ${\bf x}$ is $\sqrt{8\cdot 4}= \sqrt{32} \approx 5.65685
$.]{\begin{tikzpicture}[scale=0.55]
\begin{axis}[
unit vector ratio*= 1 1 1,
xlabel={$i$},ylabel={\rotatebox{270}{$x_i$}},ymin=0,ymax=15,xmin=0,xmax=20,
area style]
\addplot
[const plot,fill=cyan,draw=black]
coordinates
{(0,11) (1,7) (2,6) (3,6)
(4,6) (5,4) (6,4) (7,4)
(8,3) (9,3) (10,2) (11,2) (12,1) (13,1) (14,1) (15,0)}
\closedcycle;
\addplot coordinates
{(0,4) (8,4) (8,0)} \closedcycle;
\end{axis}
\end{tikzpicture}} \hspace{10mm}
\subfigure[The $w'$-index of ${\bf x}$ is $\sqrt{\frac{22}3 \cdot 11} \thickapprox 8.98146.
$]{\begin{tikzpicture}[scale=0.55]
\begin{axis}[
unit vector ratio*= 1 1 1,
xlabel={$i$},ylabel={\rotatebox{270}{$x_i$}},ymin=0,ymax=15,xmin=0,xmax=20,
area style]
\addplot
[const plot,fill=cyan,draw=black]
coordinates
{(0,11) (1,7) (2,6) (3,6)
(4,6) (5,4) (6,4) (7,4)
(8,3) (9,3) (10,2) (11,2) (12,1) (13,1) (14,1) (15,0)}
\closedcycle;
\addplot coordinates
{(0,7.3333) (11,0)} \closedcycle;
\end{axis}
\end{tikzpicture}} \hspace{10mm}
\subfigure[The $c'$-index of ${\bf x}$ is $\thickapprox 6.38691029$.]{\begin{tikzpicture}[scale=0.55]
\begin{axis}[
unit vector ratio*= 1 1 1,
xlabel={$i$},ylabel={\rotatebox{270}{$x_i$}},ymin=0,ymax=15,xmin=0,xmax=20,
area style]
\addplot
[const plot,fill=cyan,draw=black]
coordinates
{(0,11) (1,7) (2,6) (3,6)
(4,6) (5,4) (6,4) (7,4)
(8,3) (9,3) (10,2) (11,2) (12,1) (13,1) (14,1) (15,0)}
\closedcycle;
\pgfplotsset{samples=500}
\addplot+[domain=0:6.465] {0.975900073*sqrt(41.8-x^2)}
\closedcycle;
\end{axis}
\end{tikzpicture}}
\caption{Three scale-invariant indices for ${\bf x} = (11,7,6,6,6,4,4,4,3,3,2,2,1,1,1)$.} \label{fig:SIIndices}
\end{figure}


To formulate a scale-invariant version $e'$ of the Egghe index $e$, we modify the \emph{alternate} version of the definition.\footnote
{This version appears immediately after the standard one, Definition \ref{D:EggheX}, in the previous section.}  Substituting  $h'$ for $h$, in that version, we obtain: $e'({\bf x})=$ the maximum value,  taken over all citation records ${\bf x'}$ cumulatively dominated by ${\bf x}$, of the scale-invariant Hirsch index $h'({\bf x'})$.  A pathological effect, however, arises from the combination of scale invariance with the ability to shift citations from more highly cited papers to less highly cited ones, and to ``new" papers beyond the number indicated by ${\bf x }$.  The maximum value of  $h'({\b x'})$ 
will always be achieved via ${\bf x'} = (1, 1, 1, \dots , 1)$ with $l({\bf x'})$ equal to the total number $\sum\limits_{i=1}^{l({\bf x})} x_i $ of citations recorded by all papers in the citation record, so that the value $e'({\bf x})$ reduces to the square root of this total. In particular, then, $e'({\bf x})$ incorporates all citations from the horizontal tail (as well as all from the vertical tail), which is not at all what the Egghe index was supposed to do.

Could some other approach yield a credible scale-invariant version of Egghe's index?  Of course, we cannot rule out this possibility. However, it seems unlikely, and we take the pathological behavior of $e'$ as evidence that scale invariance is fundamentally incompatible with Egghe's index.  More broadly, the behavior suggests that the introduction of scale invariance may make sense only for indices closely related to the shape indices (Definition \ref{ShapeIndex}; see related comment in Section \ref{SEC:MoreQuestions}).

\medskip


\subsection{Some axioms}\label{SomeAxioms}

We propose axioms that capture certain key properties of a scientific impact index. 

\medskip

\noindent AXIOM 1: Monotonicity (Mon). We say that one citation record ${\bf x}$ is \emph{dominated} by a second record ${\bf y}$, if $l({\bf x}) \leq l({\bf y})$ and $x_k \leq y_k$ for each $k$ with $1 \leq k \leq l({\bf x})$, writing ${\bf x} \preceq {\bf y}$ to denote this situation;  ${\bf x}$ is \emph{strictly dominated} by ${\bf y}$, written ${\bf x} \prec {\bf y}$,  if ${\bf x} \preceq {\bf y}$ and ${\bf x} \neq {\bf y}$.  \emph{Monotonicity} of a scientific citation index $g$ now requires for all ${\bf x, y} \in X$ that
\begin{equation}\label{E:Mon}
\text{if} \quad {\bf x} \preceq {\bf y}, \quad  \text{then} \quad g({\bf x}) \leq g({\bf y}).
\end{equation}
\\
Adding a citation (or a new publication that has been cited) never decreases the value of a monotonic index, implying that the index of a scientist will never decrease over time.
Some authors consider monotonicity as a \emph{conditio sine qua non} for citation indices.

\medskip

\noindent AXIOM 2: Symmetry (Sym). Let ${\bf x} = (x_1,x_2, \dots,x_{l({\bf x})})$ be the citation record of an individual researcher $R$ with associated compact set $\mathcal B({\bf x})$. Let ${\mathcal B}^{\ast}({\bf x})$ be the set obtained by reflecting $\mathcal B({\bf x})$ about the line $y=x$; that is,
$(a,b) \in {\mathcal B}^{\ast}({\bf x})$ \emph{if and only if} $(b,a) \in {\mathcal B}({\bf x})$. Let ${\bf x^{\ast}}$ be the citation record obtained from
${\mathcal B}^{\ast}({\bf x})$. We call ${\bf x^{\ast}}$ the \emph{dual} citation record of ${\bf x}$, since ${\bf x^{\ast \ast}} = {\bf x}$.
The symmetry axiom requires of a scientific citation index $g$ that:
    \begin{equation}\label{E:Sym}
    g({\bf x}) = g({\bf x^{\ast}})
    \end{equation}
    for all ${\bf x} \in X$. \\

    \medskip

\noindent Examples of dual records include: \\
    If ${\bf x} = (8,6,2)$ then ${\bf x^{\ast}} = (3,3,2,2,2,2,1,1)$. \\
    If ${\bf x} = (8,6,6,2)$ then ${\bf x^{\ast}} = (4,4,3,3,3,3,1,1)$. \\
    If ${\bf x} = (13,11,11,10,7,4,3,3,3,1)$ then ${{\bf x^{\ast}}} = (10,9,9,6,5,5,5,4,4,4,3,1,1)$. \\

This condition forbids the index from having any built in bias favoring productivity (on the $x$-axis) over impact (on the $y$-axis) or impact over productivity.
While the symmetry axiom requires an equal treatment of productivity and impact by the index, it \emph{does not presuppose} (for any particular input citation record ${\bf x}$, which may itself not be symmetric) that a unit of one variable need be treated as equal to a unit of the other. Both Hirsch and Woeginger indices verify symmetry, while it fails for the Egghe index (which has value $3$ for ${\bf x} = (8,6,2)$ but $2$ for ${\bf x^{\ast}}$).

%

\medskip

\noindent AXIOM 3: Scale-Invariance (SInv) and Strong Scale-Invariance (SSInv). We consider two ways of modifying a given citation record ${\bf x} = (x_1, \dots,x_{l(\bf x)})$ via replication.  For any $k \in \mathbb{N}$, $k{\bf x}$ denotes the citation record $(kx_1, \dots,kx_{l(\bf x)})$, with the same number of publications and with the number of citations of each work multiplied by $k$.  For any $m \in \mathbb{N}$, ${\bf x}^{\leftrightarrow m}$ denotes the citation record $(\underbrace{x_1, \dots, x_1}_m, \underbrace{x_2, \dots, x_2}_m, \dots, \underbrace{x_{l(\bf x)}, \dots, x_{l(\bf x)}}_m)$, with length $m l(\bf x)$ in which each publication with its number of citations is replicated $m$ times.

The scale-invariance axiom requires of a scientific citation index $g$ that:
\begin{equation}\label{E:SI}
 g({\bf x}) \leq g({\bf y}) \quad  \text{if and only if} \quad g(k \, {\bf x}^{\leftrightarrow m}) \leq g(k \, {\bf y}^{\leftrightarrow m})
\end{equation}
for each ${\bf x, y} \in X$.  The strong version of the axiom requires, instead, the equality:
\begin{equation}\label{E:SSI}
 g({\bf x}) \cdot g(k \, {\bf y}^{\leftrightarrow m}) = g({\bf y}) \cdot g(k \, {\bf x}^{\leftrightarrow m})
\end{equation}
These scale-invariant axioms are of fundamental importance in our paper, with all of Section \ref{SEC:ScaleInvariantJustification} devoted to an analysis of their significance.

\medskip

Our next axiom, Max-Bounded, is the one that selects the scale invariant Hirsch index $h'$ from a broader class of scale invariant indices.  Immediately after stating the axiom, we'll dissect its relationship to  \emph{Nash's IIA} (see fn \ref{Salles}), the second consequential axiom used by Nash to characterize his solution to the two-person bargaining problem.

\noindent AXIOM 4: Max-Bounded (MaxB). Let ${\bf x}$ and ${\bf y}$ be two  citation records and ${\bf z}$ be defined as their componentwise maximum:
$$
z_i = \left\{
        \begin{array}{ll}
          \max \{x_i,y_i \}, & \hbox{if} \; \; i \leq \min \{l({\bf x}),l({\bf y})\} \\
          x_i, & \hbox{if} \; \; l({\bf y}) < i \leq l({\bf x})  \\
          y_i, & \hbox{if} \; \; l({\bf x}) < i \leq l({\bf y})  \\
        \end{array}
      \right.
$$
The max-bounded axiom asserts, for every  choice of ${\bf x}$ and ${\bf y}$, that:
\begin{equation}\label{E:MB}
g({\bf z}) \leq \max \{ g({\bf x})  , g({\bf y})  \}.
\end{equation}

\noindent Note that in the presence of \emph{Mon}, condition \eqref{E:MB}  is equivalent to   $g({\bf z}) = \max \{ g({\bf x})  , g({\bf y})  \}$ and also to  $g({\bf z}) \in \{ g({\bf x})  , g({\bf y})  \}$.

Recall that in the context of Nash's Bargaining Theorem, a bargaining solution $\phi$ is a function that selects a unique point $\phi (F)$ out of each feasible subset of $\mathbb{R}_{++} ^2$.
 \emph{IIA} then asserts that whenever $G \subseteq F$
 are two feasible sets, if $\phi (F) \in G$, then $\phi (G) = \phi( F)$.  Uniqueness of the point $\phi (F)$, however, rests on the requirement that any feasible set $F$ be convex (whence any real-valued, continuous and convex function $\rho$ is maximized at a unique point in $F$).

 One important difference from Nash's context, then, is that the subsets
 $\mathcal{B}({\bf x }) \subseteq \mathbb{R}_{++} ^2$ of concern to us here are typically not  convex.\footnote{
 Existence of a unique $\rho$-maximizing point in $F$ rests on additional conditions (other than convexity) on $F$, such as compactness, but these conditions are all satisfied by $\mathcal{B}({\bf x })$.
 }
So it makes sense to consider the following ``multivalued" version \emph{MVIIA} of Nash's axiom:

\begin{definition}
[Multi-Valued IIA, aka MVIIA]
 Let $\phi$ select a nonempty subset $\phi (F)\subseteq F$
out of each set  $F$ belonging
to some specified collection $\mathfrak{F}$  of subsets $F$ of  $\mathbb{R}_{++} ^2$.
 \emph{MVIIA} then asserts that whenever $G \subseteq F$
 are two sets in $\mathfrak{F}$, if $\phi (F) \cap G$ is nonempty, then $\phi (G) = \phi( F)\cap G$; that is, $\phi$ selects from $G$ those points that were selected from $F$ and remain available in $G$, providing at least one such point remains available.

\end{definition}

\noindent The MVIIA property may be of independent interest, as it seems to be related to two other well known principles: the \emph{Weak Axiom of Revealed Preferences} and the \emph{Reinforcement Axiom}, 
which arise in different subfields of the mathematical social sciences. We sketch those connections in the Appendix. 

We'd like to argue, then, that the \emph{MaxB} axiom follows from \emph{MVIIA} with $\mathfrak{F}$ being the set of bar-graph subsets, but any such argument must bridge a second important difference from Nash's context; a solution of the bargaining problem chooses a point $\phi(F) = (x_1, x_2)$ in $\mathbb{R}_{++} ^2$, while a citation index chooses a single real number $g(\mathcal{B}({\bf x})) = g({\bf x})$.  Of course, the bargaining solution characterized by Nash's axioms in fact chooses the point $(x_1, x_2)$ that maximizes a certain function $\rho\! : \mathbb{R}^2 \rightarrow \mathbb{R}$, with the actual function being $\rho (x_1 , x_2 ) = x_1 x_2$.  The scale invariant Hirsch index we characterize is given by the maximal value achieved by the very same function $\rho$ on $\mathcal{B}({\bf x})$.  
It should not be surprising, then, that to derive our \emph{MaxB} axiom from \emph {MVIIA} we must account for this translation in context:

\medskip

\noindent AXIOM 4B: IIA for citation indices (IIACI).
A citation index $g$ is  \emph{point-induced} if $g$ can be written as  $g = \rho \circ \phi$ where $\phi$ selects one or more points $(x_1 , x_2 )$ from each $F$ in the collection $\mathfrak{B}$  of all subsets of the form $\mathcal{B}({\bf x }) \subseteq \mathbb{R}_{++} ^2$, $\rho\! : \mathbb{R}^2 \rightarrow \mathbb{R}$, and $\rho$ has the same value on each point in $\phi (F)$, for each $F \in \mathfrak{B}$.  Then $g$ satisfies IIACI if $g$ is point induced via some $\phi$ satisfying MVIIA.

\medskip


The precise result, then, is as follows:

\begin{proposition}\label{IIAandMaxB}
Let $g$ be any scientific citation index satisfying IIA for citation indices.  Then $g$ satisfies Max-Bounded.
\end{proposition}

\begin{proof}
Let $g = \rho \circ \phi$ be a citation index, point induced by $\rho$, and $\phi$ satisfy \emph{MVIIA} with $\mathfrak{F}$ equal to the set $\mathfrak{B}$ of all possible bar-graph subsets $\mathcal{B}({\bf x})$ generated by citation records ${\bf x}$.  Note first that if $\mathcal{B}({\bf x}), \mathcal{B}({\bf y}) \in \mathfrak{B}$, and ${\bf z}$ is the componentwise maximum of ${\bf x}$ and ${\bf y}$ then $\mathcal{B}({\bf x}) \cup \mathcal{B}({\bf y}) \in \mathfrak{B}$, with $\mathcal{B}({\bf x}) \cup \mathcal{B}({\bf y}) = \mathcal{B}({\bf z})$. It follows that $\phi(\mathcal{B}({\bf z})) = [\phi(\mathcal{B}({\bf z})) \cap \mathcal{B}({\bf x})] \cup [\phi(\mathcal{B}({\bf z})) \cap \mathcal{B}({\bf y})] $, whence at least one of the sets $\phi(\mathcal{B}({\bf z})) \cap \mathcal{B}({\bf x}) , \phi(\mathcal{B}({\bf z})) \cap \mathcal{B}({\bf y}) $ is nonempty.  So \emph{MVIIA} tells us that $\phi(\mathcal{B}({\bf x})) = \phi(\mathcal{B}({\bf z})) \cap \mathcal{B}({\bf x})$ or $\phi(\mathcal{B}({\bf y})) = \phi(\mathcal{B}({\bf z})) \cap \mathcal{B}({\bf y})$.  Thus g({\bf z}) = $\rho (\phi(\mathcal{B}({\bf z})))$ must be equal either to $g({\bf x}) = \rho (\phi(\mathcal{B}({\bf x})))$ or to $g({\bf y}) = \rho (\phi(\mathcal{B}({\bf y})))$, whence $g({\bf z}) \in \{g({\bf x}) , g({\bf y})   \}$.
\end{proof}

\noindent The normative content of IIACI seems clear: the overall measure of effectiveness of a scholarly record ${\bf x}$ is determined by the value of some single ``best" entry $x_j$ in that record.  This best $x_j$ corresponds to the column of $\mathcal{B}({\bf x})$ containing the point $\phi(\mathcal{B}({\bf x}))$.  Proposition \ref{IIAandMaxB} tells us that the MaxB axiom inherits that normative justification.

Our last three axioms govern the way an index begins to respond to increased publications and citations, as well as how it continues to respond.  These three should be thought of as a package.

%
%
%


%

 \vspace{4mm}

\noindent AXIOM 5: Weak Responsiveness (WResp).
\begin{equation} \label{E:SR}
    g(2,2) > g(1).
    \end{equation}

     \noindent AXIOM 6: Square Root Responsiveness (SqrtResp).
\begin{equation} \label{E:SQR}
    g(2)=\sqrt{2}.
    \end{equation}

          \noindent AXIOM 7: Linear Growth (LGr).\label{E:LGr} Let $R$ be a researcher whose publication and citation histories follow the \emph{simple deterministic model} (see Section \ref{SEC:INTRO}) with positive integer valued parameters $p$ and $c$; and let $g_R(n)$ be the value of index $g$ for $R$ after $n$ years of publication, with $n \in \mathbb{N}$.     Then the points $(n,g_R(n))$, $n \in \mathbb{N}$, all lie within a strip bounded by two parallel straight lines.

\vspace{4mm}

 Our axiomatic characterization, Theorem \ref{T:char} in Section \ref{SEC:Axiomatization}, uses the first four axioms to characterize the class $\{h_a\}_{a>0}$ of all Hirsch powers.  Adding axioms $5$ (WResp) and $7$ (LGr) together then characterizes the scale-invariant Hirsch index $h' =h'_{\frac{1}{2}} $ in particular; this is part $(ii)$ of Theorem \ref{T:char}, aka the \emph{Main characterization}.  Alternately, in part $(iii)$ of the theorem, aka the ``Strong" characterization, axiom $6$ (SqrtResp) alone substitutes for the combination of $5$ and $7$.

 Why have we chosen to offer two alternate characterizations for $h'$? One virtue of the \emph{Main characterization} is that \emph{all} of the axioms hold for the original Hirsch index, except of course for the scale invariance axiom itself, so that scale invariance alone is what distinguishes the two.\footnote{It would be desirable to have an alternative to the scale invariance axiom---an opposing principle with equally clear normative content---that would characterize the original Hirsch index, when substituted for Strong Scale Invariance in the Main Characterization.}
 But we see  \emph{normative transparency} as the principal advantage of the Main characterization; \emph{all} of the axioms have clear normative content.  We've already discussed the meaning of axiom $4$, as well as that of axiom $7$ along with the role it played in Hirsch's original formulation of his index.
Axiom $5$ also has a clear meaning, asserting that a citation record of 2 publications with 2 citations each is superior to one with a single publication having a single citation; this axiom seems to be satisfied by most
of the many indices that have been proposed since Hirsch's paper.  Note also that  the constant function $1$ (all citation records lead to an index of $1$) satisfies all the axioms listed here except for $5$ and $6$.  Of course, when coupled with the other axioms, (WResp)'s effect on responsiveness is greatly magnified; we might say, then, that this axiom ``kick-starts" the responsiveness of citation indices.

Including Linear Growth among the axioms in the Main characterization is arguably an attractive feature, given its role in the formulation of Hirsch's original index.  However, it seems that much of the force of this axiom is not actually needed for the characterization.  We see this from the Strong characterization, which replaces the combination of Linear Growth and Weak Responsiveness with the single axiom of Square Root Responsiveness, asserting $g(2) = 2^{\frac{1}{2}}$. It seems that the Strong Characterization achieves the same results using weaker axioms.

There is a cost, however, in loss of normative transparency; Square Root Responsiveness seems quite technical in flavor, lacking the clear interpretations of the other axioms. To see what  Square Root Responsiveness actually tells us about an index, we can compare it with an alternative called \emph{Scale Responsiveness}, asserting $g(2) > g(1)$.  In strength, this alternative falls between Weak Responsiveness and Square Root Responsiveness. With the help of Monotonicity, it implies Weak Responsiveness.  But it
already encapsulates a kernel of scale invariance, and in particular the Hirsch, Woeginger, and Egghe indices all fail to satisfy it.  Scale responsiveness is equivalent, of course, to the requirement that

\begin{equation} \label{E:Nor}
    g(2) = 2^a
    \end{equation}

\noindent for \emph{some} real number $a > 0$, while Square Root Responsiveness pins the value of $a$ at $\frac{1}{2}$.  Thus, we can think of Square Root Responsiveness as a variant of Weak Responsiveness that is stronger in two respects.

\section{Incomparability of scale}  \label{SEC:ScaleInvariantJustification}

Before the advent of indices such as Hirsch's, it was more common to use, as a metric,
the sum of the number of citations accrued by each publication, which of course is equal
to the total area under the step-function. Hirsch argues in \cite{Hir05} that the area in the two `tails'
of the $P \times I$ graph (see Figure~\ref{fig:BasicIndices}(a)) should not be counted.
A large vertical tail may arise from a small number of papers having a number of citations that is uncharacteristically high
(compared to other publications by the same author), perhaps because they were co-authored by
some particularly distinguished co-author; publications co-authored with one's Ph.D. thesis
advisor, and based on that thesis work, would be an important special case. A large horizontal
tail may represent a substantial number of publications that receive a small number of
citations apiece, presumably indicative of the low impact these publications have had on
the field. The vertical tail represents the part of one's record that has high impact with low
productivity, while the horizontal tail corresponds to high productivity with low impact.

By counting only that part of the area contained in a convex sub-region, Hirsch's index
lops off these two tails (as does Woeginger's, using a somewhat different way of deciding what
constitutes the tail); the resulting metric counts only that part of a publication record that
reflects a suitable  between productivity and impact. The question then becomes,
`What, exactly, should suitable mean, in this context?' The Hirsch square and the Woeginger triangle are each symmetric about the line $y = x$, so when either is used as the convex region, the two tails are truncated at the same point.  Yet the arguments for discarding citations in the vertical tail is quite different from that for the horizontal tail, and Hirsch never argues that
these tails resemble (or should resemble) one another in shape or area.   We see no good argument, then, that the
two tails should always be treated in the same way.

In fact, the work of \cite{FHLB} suggests a systematic tendency for vertical tails to be larger, a trend that is even stronger among Nobel prize winners.  A particularly extreme example is that of Professor John Forbes Nash, the only scholar in the history who won both a Nobel Prize (in Economics) and the Fields medal in Mathematics.   Google Scholar states that Nash's papers have received 21,690 citations in total.\footnote{As of August, 2022.  There is almost surely some miscounting in the Google Scholar record.  For example, some of the less cited articles appear to be about him, rather than by him, but the qualitative picture seems accurate.}    But his two most cited papers (\emph{The bargaining problem}---a principal reference for our work here---and
\emph{Equilibrium points in n-person games}) together account for over 20,000 of those citations, making for a vertical tail that contains 99.36$\%$ of his total citations, a Hirsch square containing only 0.56$\%$, and a horizontal tail with 0.01$\%$.   Nash's case is atypical, of course, but an argument can be made that too many scholars are losing too many  citations to the tails, with most of the loss to the vertical tail.  Perhaps the solution is to accord privileged treatment to the vertical tail, for example the way Egghe's index does, or by replacing the Hirsch square with a vertical rectangle of specified proportions.  Yet Fenner et al. \cite{FHLB} also point out that among Nobelists, some of the citation records with the very highest index values actually have larger horizontal tails.

Moreover, \emph{any}
 choice of a particular shape
of fixed proportions for the convex region (whether or not that region is symmetric about $y = x$) constitutes a commitment to a single notion of
suitable balance, to be imposed on all publication records, in all fields and subdisciplines.
Fixing the proportions of such a shape is tantamount to fixing the ratio between a unit on the vertical (impact) axis and a unit on the horizontal (productivity) axis. As an analogy, imagine that we have designed a numerical metric of performance for sports sedans, which
we use to rank different models, and which relies on measurements taken with particular
units, such as feet (for distance) and minutes (for time). It awards a higher rating to the BMW
320i than to the Infiniti Q50, but when we take the same measurements using different units
(of meters and minutes, say) and combine them in a like manner, the relative scores reverse,
with the Infiniti outscoring the BMW. Would we trust such a metric to provide meaningful
comparisons?

We argue that a similar flaw arises in any citation index that relies on a fixed choice of
shape with fixed proportions.
The Hirsch region, for example, is a square; by requiring the two sides to be equal in length, a square
equates a unit on the horizontal axis (a single publication) with a unit on the vertical
axis (a single citation). But these are completely different sorts of objects; why should one
publication be equated with one citation? This reliance is apparent, as well, in the earlier formula (1) $s_{h,R} = \dfrac{pc}{p+c},$
for slope of the Hirsch index, in which the denominator of $p+c$ sums two  quantities measured in unrelated units, somewhat like adding meters to seconds.
The assumption that a vertical unit should be treated as the same size as a horizontal
unit is a special case of the more general presumption that there should be some fixed
ratio $r$ relating the unit sizes of the axes; we will refer to this broader version as the scale
comparability assumption, and to any index that rests on such an assumption as a \emph{fixed
scale citation index}.

It is a common observation that the Hirsch index, or any other numerical measure of scholarly effectiveness, should not be used to compare two scholars from different disciplines. We argue that fixed scale indices such as $h$, $w$ or $c$ can introduce distortions, even when used to compare researchers from the \emph{same} discipline or subdiscipline, because disciplines differ in their \emph{publication culture}, and the citation records of individual researchers in a discipline differ as to how well they fit the publication culture of their discipline. Here by \emph{publication culture}  we refer to a variety of factors that differ among disciplines, and among subdisciplines, such as:
\begin{enumerate}
\item Is it more common to publish fewer, longer papers or a greater number of shorter works?
\item Is it more common to cite only a small number papers (perhaps limiting the bibliography to that part of the literature most directly called upon) or a greater number?
\item Does the field have a large number of researchers, or very few?
\item Among researchers who are actively publishing in the discipline, what are the average proportions of the rectangle of greatest area that can be inscribed in $\mathcal{B} ({\bf x})$?
\end{enumerate}
Clearly, the first three items above help drive the last.

For a fixed scale index such as $h$, a change in the units with which $p$ and $c$ are measured can switch which of two scholars has the greater index value, as in the car metaphor mentioned earlier. Such \emph{ranking inversions}, discussed below, are perhaps the most striking of these distortions.
Our original motivation for introducing scale invariance was that it promised to suppress these unintended and undesirable consequences of imposing a common fixed scale over disciplines that differ in publication culture. Indeed,  we'll show that the scale-invariant indices discussed here are free from such types of distortion.

An extreme example may help illustrate these points. Consider disciplines wherein the
publication culture results in some publication records that consist of book-length
monographs, which require more time to write and are presumably fewer in number, but
which may be expected to have correspondingly greater impact in the form of more citations.
Suppose a mid-career scholar has authored four books, each of which has been cited
hundreds of times. A second scholar in the same discipline has written seven books, each
of which has been cited from ten to twenty times. It is easy to see that the first scholar
has Hirsch index of $4$ while the second has Hirsch index of $7$; the $h$-index is blind to the
substantial difference in impact between the works of the first and second scholar, because
the Hirsch square inscribed within $\mathcal B({\bf x})$ bumps into a horizontal limit (due to the number of vertical bars)
well
before any opportunity for the heights of those bars to have any limiting effect.
Equating one book
with one citation is surely unreasonable.

This situation is extreme, and so one might argue
that the Hirsch index\footnote{The Woeginger index and the $c$-index behave similarly in this situation.} was never intended to cope with such wide variance in publication
culture.
Milder variations in publication culture, however, yield effects that, while smaller,
are similarly nocuous. An underlying problem with any fixed scale index can be summarized as follows:

\vspace{3mm}

\noindent \emph{When a fixed scale index is used to compare two scholars in the same subdiscipline, the one whose citation record varies more from the fixed scale assumption of that index is relatively disadvantaged, even if that additional variance places  them closer to the norms for that subdiscipline}.

\vspace{3mm}

Algebraic topology, for example, tends towards publication of fewer, lengthier papers than is typical for some other subfields of Mathematics.  Figure \ref{fig:Stretch} shows the (hypothetical) bar graph representations $\mathcal B({\bf x}_C)$, $\mathcal B({\bf x}_D)$ of two early career algebraic topologists $C$ and $D$.  While $D$ has a citation record that is more typical for her field, her Hirsch index is only $4$; $C$ achieves a higher index value of $6$ because his record  fits the implicit fixed scale assumption of the index better (in that less area is lost to the tails).   

A related consequence of fixed scale is \emph{ranking inversions}, wherein a pair
of scholars in the same discipline are ranked one way, while an arguably analogous pair in
a different discipline are ranked oppositely by the same index. Finite combinatorics leans in the opposite direction from algebraic topology, with more publications that are shorter. In part this divergence arises from fundamental differences between the two sub-disciplines. 


Imagine that the same two individuals $C$ and $D$ had instead become finite combinatorists---we'll call them $C^\prime$ and  $D^\prime$. We might expect the bar graphs for $C^\prime$ and  $D^\prime$ to resemble horizontally stretched versions of those for $C$ and $D$. One can easily apply arbitrary stretches to the horizontal axis, but if we require  all stretched bar graphs to arise from actual citation records, then one can only apply integer stretch factors, which is equivalent to repeating each coordinate $x_i$ (of the citation record ${\bf x}$)  $m$ times, for some positive integer $m$. The right hand side of Figure \ref{fig:Stretch} shows
$\mathcal B({\bf x}_{C^\prime})$ and $\mathcal B({\bf x}_{D^\prime})$, in which the citation records of $C$ and $D$ have been horizontally stretched by a factor of $3$: ${\bf x}_{C^\prime} = {\bf x}_{C} ^{\leftrightarrow 3}$ and  ${\bf x}_{D^\prime} = {\bf x}_{D} ^{\leftrightarrow 3}$.  The result is an inversion, with the Hirsch index of $D^\prime$ boosted from $4$ to $11$ while the same horizontal stretch increases $C^\prime$'s index by much less: from $6$ to $8$.  We can block this effect by requiring a scale free index $g$ to respect horizontal scaling: a pair ${\bf x}, {\bf y}$ of scholarly records should satisfy $g({\bf x})<g({\bf y})$ if and only if $g({\bf x}^{\leftrightarrow m}) < g({\bf y}^{\leftrightarrow m})$.

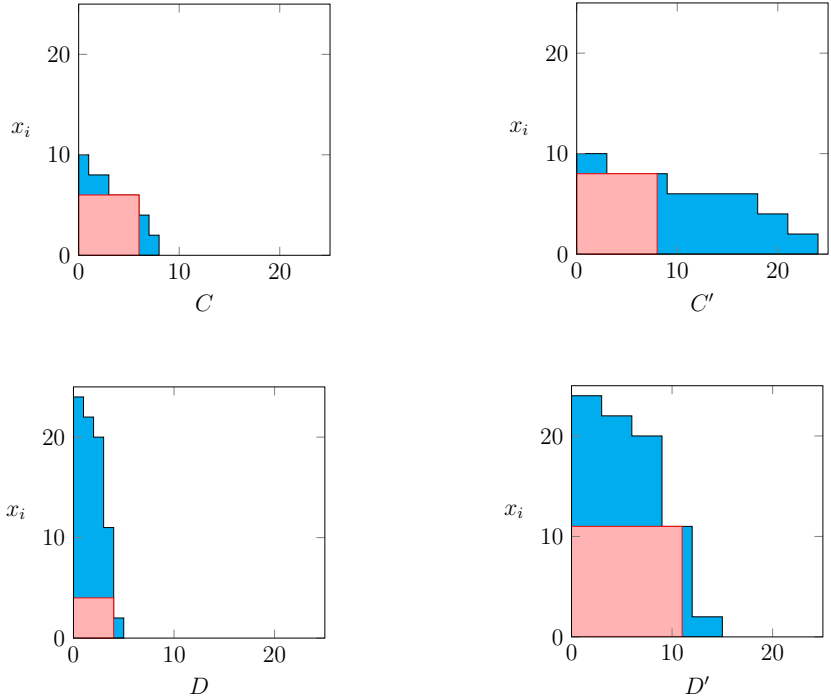
\begin{figure}[H] 
\centering
\subfigure
{\begin{tikzpicture}[scale=0.75]
\begin{axis}[
unit vector ratio*= 1 1 1,
xlabel={$C$},ylabel={\rotatebox{270}{$x_i$}},ymin=0,ymax=25,xmin=0,xmax=25,
area style]
\addplot
[const plot,fill=cyan,draw=black]
coordinates
{(0,10) (1,8) (2,8) (3,6)
(4,6) (5,6) (6,4) (7,2) (8,2)}
\closedcycle;
\addplot coordinates
{(0,6) (6,6) (6,0)}
\closedcycle;
\end{axis}
\end{tikzpicture}} \hspace{20mm}
\subfigure
{\begin{tikzpicture}[scale=0.75]
\begin{axis}[
unit vector ratio*= 1 1 1,
xlabel={$C'$},ylabel={\rotatebox{270}{$x_i$}},ymin=0,ymax=25,xmin=0,xmax=25,
area style]
\addplot
[const plot,fill=cyan,draw=black]
coordinates
{(0,10) (1,10) (2,10)
(3,8) (4,8) (5,8) (6,8) (7,8) (8,8)
(9,6) (10,6) (11,6) (12,6) (13,6) (14,6) (15,6) (16,6) (17,6)
(18,4) (19,4) (20,4)
(21,2) (22,2) (23,2) (24,2)}
\closedcycle;
\addplot coordinates
{(0,8) (8,8) (8,0)}
\closedcycle;
\end{axis}
\end{tikzpicture}}   
\vskip0.5truecm
\subfigure
{\begin{tikzpicture}[scale=0.75]
\begin{axis}[
unit vector ratio*= 1 1 1,
xlabel={$D$},ylabel={\rotatebox{270}{$x_i$}}, ymin=0,ymax=25,xmin=0,xmax=25,
area style]
\addplot
[const plot,fill=cyan,draw=black]
coordinates
{(0,24) (1,22) (2,20) (3,11) (4,2) (5,2)}
\closedcycle;
\addplot coordinates
{(0,4) (4,4) (4,0)}
\closedcycle;
\end{axis}
\end{tikzpicture}} \hspace{20mm}
\subfigure
{\begin{tikzpicture}[scale=0.75]
\begin{axis}[
unit vector ratio*= 1 1 1,
xlabel={$D'$},ylabel={\rotatebox{270}{$x_i$}},ymin=0,ymax=25,xmin=0,xmax=25,
area style]
\addplot
[const plot,fill=cyan,draw=black]
coordinates
{(0,24) (1,24) (2,24)
(3,22) (4,22) (5,22)
(6,20) (7,20) (8,20)
(9,11) (10,11) (11,11)
(12,2) (13,2) (14,2) (15,2)}
\closedcycle;
\addplot coordinates
{(0,11) (11,11) (11,0)}
\closedcycle;
\end{axis}
\end{tikzpicture}}
\caption{On the left, bar graphs ${\mathcal B}({\bf x}_{C})$ and ${\mathcal B}({\bf x}_{D})$, with ${\bf x}_C = (10,8,8,6,6,6,4,2)$, and ${\bf x}_D= (24,22,20,11,2)$. On the right, ${\mathcal B}({\bf x}_{C^\prime})$ and ${\mathcal B}({\bf x}_{D^\prime})$ (obtained from the first via a horizontal stretch with stretch factor $3$).} \label{fig:Stretch}
\end{figure}

  In comparing publication cultures across various disciplines, it is important to note, as well, that the number of active scholars in a discipline places a ceiling
on the possible number of citations. Google Scholar lists almost $12,000$ citations for
Nash's paper on the bargaining problem~\cite{Nas50}, and over $23,000$ for one of Einstein's co-authored
papers~\cite{EPR35}. The world's most renowned Egyptologist could never achieve a similarly high
number of citations; the planet lacks sufficiently many Egyptologists.
Suppose, then, that the collective citation records of
Egyptologists roughly resemble vertically compressed versions of the citation records
arising from some other discipline,
with a greater number of publishing scholars; all the heights of the bars in the bar graphs for the other discipline
have been multiplied by $4$, for example, compared to Egyptology.
Then we might want our index $g$ to respect that vertical scaling, with records  ${\bf x}, {\bf y}$,
of a pair of Egyptologists satisfying $g({\bf x}) < g({\bf y})$ \emph{if and only if}
$g(4 {\bf x}) < g(4 {\bf y})$ (where
$4 {\bf x}$ and $4 {\bf y}$ represent the records for a pair of scholars
from the other discipline whose careers are otherwise comparable to those of the two Egyptologists).

The demand that an index respect both vertical and horizontal scaling  is captured by the (SInv) axiom.\footnote{But note that in the presence of the Symmetry axiom, respect for either type of scaling entails respect for the other.}
If we are more restrictive by demanding numerical proportionality under scaling,
 the result is Equation \eqref{E:SSI} which is the (SSInv) axiom.


\subsection{Taking a cue from the Nash bargaining solution and interpersonal comparisons of utility}

In the \emph{two-person bargaining problem} considered by John F. Nash~\cite{Nas50}, players $1$ and $2$ are
attempting to agree on a point $(x,y)$ chosen from a feasible region $F$, which consists of a
closed convex region of the plane. In Figure \ref{fig:NashFig}, $F$ is the first quadrant region under the curve (boundary included).
The coordinates $x$ and $y$ of each point represent the utility payoffs to $1$ and $2$ respectively.
If the players fail to sign a binding agreement choosing one such point, then they default
to some disagreement point $(x_D, y_D)$ of payoffs in $F$, which we will take to be the origin in
our simplified version, so that a problem instance for us will be a pair $(F,(0,0))$.


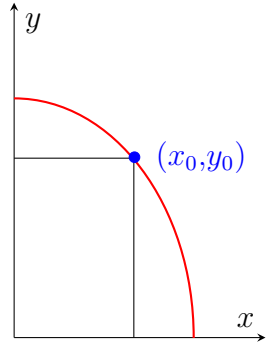
\begin{figure}[!ht]
{\centering
\begin{tikzpicture}
\coordinate (Point) at (1.6,2.4);
\begin{axis}[
unit vector ratio*= 1 1 1,
minor tick num=3,
axis y line=center,
axis x line=middle,
ymax=2.8, ymin=0,
xmax=2.1, xmin=0,
xlabel=$x$,ylabel=$y$,
xtick={10},   
ytick={10},    
]
\draw[-] (0,1.5) -- (1,1.5);
\draw[-] (1,0) -- (1,1.5);
\draw[red, thick] (0,0)
ellipse [
x radius=1.5, y radius=2];
\draw [blue,fill] (Point) circle (2pt)
node [right] {\ ($x_0$,$y_0$)};
\end{axis}
\end{tikzpicture}
\caption{The two-person bargaining problem.}  \label{fig:NashFig}
}
\end{figure}

Nash's proposed solution yields an agreement point $(x_0, y_0)$ for each such instance, and
is elegantly characterized via four axioms. Our focus is on one of these four, invariance
under linear transformations of utility scales. This axiom is motivated by a well known
assumption in utility theory, the \emph{impossibility of interpersonal comparisons of utility}, which
asserts that there is no meaning to statements such as `Sarah received more utility than
Jordi did.' A consequence is that any choice of the unit size used to measure utility for a
single player is completely arbitrary, with no linkage to the size of a unit used for some other
player. To test whether a proposed bargaining solution honors this principle, one can ask
what happens when a feasible region $F$ is transformed by halving (for example) the size of
the unit used to make all utility measurements for one of the players---player $1$ for example
(corresponding to the horizontal axis of Figure \ref{fig:NashFig}).

The effect of halving units on the $x$-axis is to double all $x$ coordinates of points in $F$,
so that the transformed version $F^{\ast}$ looks as if Figure \ref{fig:NashFig} were stretched horizontally by a
factor of 2. As the original choice of unit size had no intrinsic meaning, our point of view
is that $F$ and $F^{\ast}$ are equivalent representations of the same underlying bargaining problem
and should have the same solution. One way of expressing `same solution' is to assert that
when only the $x$ axis is rescaled, there is no effect on the $y$ coordinate of the solution: the
$y_0$ for $F^{\ast}$ is equal to $y_0$ for $F$. In the presence of Symmetry, it's equivalent to demand that when both axes are rescaled,
the new solution chooses an equivalently rescaled solution point.

Nash's solution chooses the point $(x_0,y_0) \in F$ that maximizes the product
$x_0y_0$:
\begin{equation} \label{EQ:Nash}
Nash(F) = \arg \max_{(x,y) \in F} xy.
\end{equation}
Equivalently (as in Equation \eqref{EQ:argmax-sih-a}), $Nash(F) = \arg \max\limits_{(x,y) \in F} (xy)^a$ for any fixed $a>0$.

Observe (Figure~\ref{fig:NashFig}) 
that this solution coincides with upper right corner of the inscribed
rectangle having maximal area. It is easy to see from $F$'s convexity that this maximizer is
unique. Why does it satisfy invariance under linear scaling? Any rescaling of the two axes
by factors of $k$ and $m$ transforms each rectangle of area $A$ inscribed in $F$ into a rectangle of
area $km$A inscribed in $F^{\ast}$, so that the rectangle $R^{\ast}$ of greatest area in $F^{\ast}$
is the transformed version of the greatest area rectangle $R$ in $F$. Thus the solution $(x_0,y_0)$ for $F$ is transformed into the solution $({x_0}^{\ast},{y_0}^{\ast})$ for $F^{\ast}$ by the same pair $k, m$ of scale factors that turned $F$ into $F^\star$.


\subsection{The scale-invariant property of some indices}\label{WhatInvIs}

When we attempt to transfer Nash's idea to the context of citation indices, how exactly
should the ideal of scale invariance be expressed? What new indices satisfy this ideal? The
ideal itself is most directly and easily expressed in terms of the following theorem:

\begin{theorem} \label{t:SII} Let ${\bf x}$ and ${\bf y}$ be any two citation records, $k$ and $m$ be positive integers, and $g$
be any scale-invariant symmetric shape citation index.  Then,
$$g({\bf x}) \leq g({\bf y}) \quad  \text{if and only if} \quad g(k \, {\bf x}^{\leftrightarrow m}) \leq g(k \, {\bf y}^{\leftrightarrow m})$$
\end{theorem}
Theorem~\ref{t:SII} follows from the following Lemma, whose proof follows immediately from the definition of scale-invariant symmetric shape citation index.

\begin{lemma} \label{p:SII} Let ${\bf x}$ be any citation record, $k$ and $m$ be positive integers, and $g$
be any scale-invariant symmetric shape citation index. Then,
$g(k{\bf x}) = \sqrt{k} g({\bf x})$ and $g({\bf x}^{\leftrightarrow m}) = \sqrt{m} g({\bf x})$.
\end{lemma}


Theorem~\ref{t:SII} thus suggests that the quantity invariant under scale transforms is the \emph{ranking}
induced by an index, and this is exactly what the scale invariance axiom (SInv) requires.  So why require the stronger version (SSInv) for our characterization?  If we modify $h'$ by applying a suitable monotonic transform (an arbitrary order-preserving function $f \!\! :[0, +\infty) \rightarrow [0, +\infty)$ for which $f(1)=1$ and $ f(\sqrt{2})= \sqrt{2}$), then the resulting index $f \circ h'$ satisfies most of the axioms, including scale invariance (SInv) and square root responsiveness (SqrtResp) but excluding strong scale invariance (unless $f(x)=x$ holds for each $x$ equal to the square root of some integer, in which case we would have $f \circ h'=h'$). In particular, this shows that for the alternate version (part $(iii)$) of our characterization Theorem \ref{T:char}, we cannot replace (SSInv) with the weaker version (SInv).  If $|f(x) - x|$ is bounded (as it is for $f(x) = x + sin(x)$, for instance) then $f \circ h'$ also satisfies linear growth (LGr), showing that the Main version also cannot go through with the weaker version (SInv) alone.\footnote{
Of course, $f \circ h'$ and $h'$ would rank citation records in exactly the same way, so the distinction between (SInv) and (SSInv) is arguably of limited, technical interest only.
} 

\section{Linear Growth}  \label{SEC:LinearGrowth}

In defining his index, Hirsch used the side-length of the largest inscribed square---equivalently, the \emph{square root} of the square's area, rather than the area itself. An individual $R$'s citation record ${\bf x}$ changes over time, with ${\bf x} = {\bf x}_R(t)$, and Hirsch's reasoning was based on the growth of the function $h({\bf x}_R(t))$, which gives the index value as a function of the number $t$ of years since researcher $R$'s professional career began. Under the simple deterministic model, each researcher is endowed with two parameters; these determine the number $p$ of new papers published each year and the number $c$ of citations made each year, to each paper published that year or earlier. Under this model, using the square root of the area guarantees that the graph of $h({\bf x}_R(t))$ is \emph{almost a straight line through the origin} (we'll be more precise in a moment). One can then compare the research record of early- and late-career researchers via the slopes of their lines, factoring out the advantage otherwise provided by a longer career. No one argues that this two-parameter model is at all realistic, but one is nonetheless left with the sense that comparing two researchers via index growth slopes is more sensible if the \emph{square root is applied}, lest there be a built-in quadratic advantage for the more senior scholar.

This is why the indices we propose apply a square root to the area of an inscribed shape. Consequently each of them, along with the original versions of the Hirsch and Woeginger indices, satisfies the Linear Growth axiom (Section \ref{SomeAxioms}), which makes the notion of ``almost a straight line" precise by requiring the graph of $g({\bf x}(t))$ to lie within a strip bounded by two parallel straight lines.  As we'll see, the lower line of these two passes through the origin, and the strip is quite narrow, relative to the dimensions of the entire graph (over a $20-40$ year publishing career).

\begin{proposition}\label{LGrProp}
Under the simple deterministic model with integer parameters $p$ and $c$, each of the indices $h$, $h'$, $w$, and $w'$ satisfy the Linear Growth axioms: if $g$ is any of these four indices, the points on the graph of $g({\bf x}(t))_{t \in \mathbb{N}}$ all lie in the strip between two lines with a common slope $s_g$, with the lower line passing through the origin and the upper line spaced a distance  $d_g$ above it.  The values of  $s_g$ and  $d_g$ for these indices are as follows:
\begin{itemize}
\item $s_h = \dfrac{pc}{p+c} = d_h$
\item $s_{h'} = \dfrac{\sqrt{pc}}{2} = d_{h'}$
\item $s_w = \min \{p,c\};\hspace{2mm} d_w = 0$
\item $s_{w'} = \sqrt{pc};\hspace{2mm}  d_{w'} = 0$
\end{itemize}
%
%
%
%
%
%
%
%
%
\end{proposition}

\noindent Note that $s_h$ and $s_{h'}$ are equal \emph{if and only if} $p=c$, with $s_h < s_{h'}$ when $p \neq c$.  Arguments similar to those that follow for $h'$ show that any other symmetric scale invariant shape index also satisfies the Linear Growth axiom, with slope equal to some scalar multiple of $\sqrt{pc}$.

\begin{proof} Figure \ref{BofxFig} shows (as a stepped solid line) the graph of the correspondence $c_{{\bf x}(t)}$ after a $t$-year career, for a researcher with integer parameters $p$ and $c$ under the simple deterministic model; $\mathcal{B}({\bf x}(t))$ is the first quadrant region under this graph.  Note that the leftmost (highest) vertical strip of this region has width $p$ and height $ct$; this strip consists of $p$ separate bars of $\mathcal{B}({\bf x}(t))$ (each of width $1$, height $ct$) pushed together, and its area is the total number $pct$ of citations made, over the $t$-year career, to the $p$ papers published in year $1$ of the researcher's career.  The dotted line  $\overline{QR}$ and dashed line $\overline{Q^\star R^\star}$ form the hypotenuses of the inscribed and circumscribed triangles $\triangle QOR$ and $\triangle Q^\star O R^\star$, respectively, with   $\triangle QOR \subseteq \mathcal{B}({\bf x}(t)) \subseteq \triangle Q^\star O R^\star$.  The equations of these lines are:
\begin{equation}\label{inscr}
\overline{QR}: y = -  {\Big (}\frac{c}{p} {\Big )}x+ct
\end{equation}
and
\begin{equation}\label{circum}
\overline{Q^\star R^\star}: y = - {\Big (}\frac{c}{p} {\Big )}x+c(t+1)
\end{equation}

\begin{figure}[!ht]
{\centering
\includegraphics[height=71mm,width=93mm] {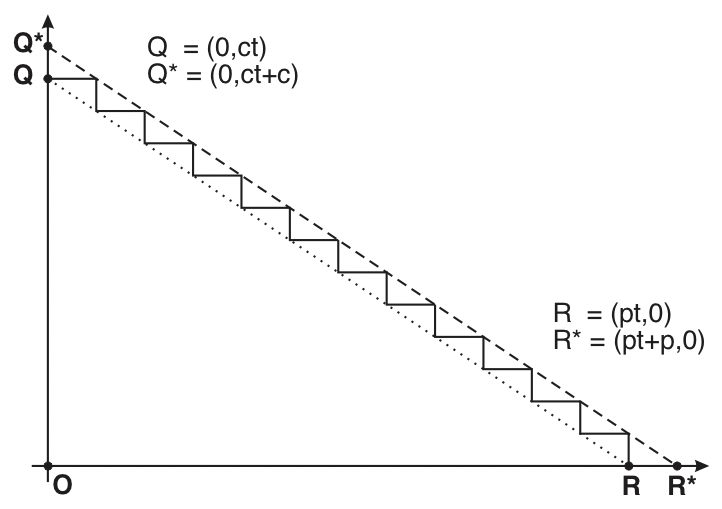}
\caption{$\mathcal{B}({\bf x}(t))$ under the simple deterministic model.}\label{BofxFig}}
\end{figure}

\noindent \emph{Argument for} $h$: Let $S(t)$, $h({\bf x}(t))$, and $S^\star (t)$ denote the side lengths of the largest squares inscribed (with one vertex at the origin) in regions $\triangle QOR$,
$ \mathcal{B}({\bf x}(t))$
and $\triangle Q^\star O R^\star$
respectively.  Thus

\begin{equation}\label{RunIneq}
S(t) \leq h({\bf x}(t)) \leq S^\star (t).
\end{equation}
We solve for $S(t)$ as the $x$ coordinate of the intersection of $\overline{QR}$ with the line $y=x$ , finding
\begin{equation}\label{SLine}
S(t) = {\Big(} \frac{pc}{p+c}  {\Big)}t.
\end{equation}
Similarly,
\begin{equation}\label{SStarLine}
S^\star (t) = {\Big(} \frac{pc}{p+c}  {\Big)}t + {\Big(} \frac{pc}{p+c} {\Big)}
\end{equation}
Inequality \eqref{RunIneq} now tells us that the graph of $h({\bf x}(t))$ lies between the two parallel lines with Equations \eqref{SLine} and \eqref{SStarLine}, which have a common slope of $\frac{pc}{p+c}$ and are spaced apart vertically by the same quantity $\frac{pc}{p+c}$.   Over a thirty-year career, the abcissa of the $h({\bf x}(t))$ graph (for $t \in \mathbb{N}$) will grow from $0$ to around $\frac{30pc}{p+c}$, a change thirty times the size of the vertical separation between the two lines bounding the graph, showing that the points on the graph lie ``almost along a straight line" of slope $\frac{pc}{p+c}$.

\medskip

\noindent \emph{Argument for} $h'$: Let $S'(t)$, $h'({\bf x}(t))$, and $S'\hspace{0.25mm}^\star (t)$ denote the square roots of the areas of the rectangles of maximal area inscribed (with one vertex at the origin) in regions $\triangle QOR$,
$ \mathcal{B}({\bf x}(t))$
and $\triangle Q^\star O R^\star$
respectively.  Thus

\begin{equation}\label{RunIneqPrime}
S'(t) \leq h'({\bf x}(t)) \leq S'\hspace{0.25mm}^\star (t).
\end{equation}

For these maximal area inscribed rectangles, the vertex diagonally opposite to the origin is at the midpoint of the hypotenuse, which is ${\big (} \frac{ct}{2}, \frac{pt}{2}{\big )} $ for $\overline{QR}$ and ${\big (} \frac{c(t+1)}{2}, \frac{p(t+1)}{2}{\big )} $ for $\overline{Q^\star R^\star}$. From this, one obtains

\begin{equation}\label{SLinePrime}
S'(t) = {\Big(} \frac{\sqrt{pc}}{2}  {\Big)}t,
\end{equation}
and
\begin{equation}\label{SStarLinePrime}
S'\hspace{0.25mm}^\star (t) = {\Big(} \frac{\sqrt{pc}}{2}  {\Big)}t + {\Big(} \frac{\sqrt{pc}}{2}{\Big)}.
\end{equation}
Inequality \eqref{RunIneqPrime} now tells us that the graph of $h'({\bf x}(t))$ lies between the two parallel lines with Equations \eqref{SLinePrime} and \eqref{SStarLinePrime}, which have a common slope of $\frac{\sqrt{pc}}{2}$ and are spaced apart vertically by the same quantity $\frac{\sqrt{pc}}{2}$.   An argument just like that for $h$ now justifies our claim that points on the $h'({\bf x}(t)) $ graph lie ``almost along a straight line" of slope $\frac{\sqrt{pc}}{2}$.

\medskip

\noindent \emph{Argument for} $w$: The largest isosceles right triangle fitting inside $\mathcal{B}({\bf x}(t))$ (with right angle at the origin) will have two legs of length of $|\overline{OQ}| = ct$ or of length $|\overline{OR}| = pt$, whichever is smaller.  It follows that $w({\bf x}(t))=\emph{min}(p,c)t$.

\medskip

\noindent \emph{Argument for} $w'$: Among right triangles fitting inside $\mathcal{B}({\bf x}(t))$ with right angle at the origin, the one with maximal area is clearly $\triangle QOR$ itself, with leg lengths $|\overline{OQ}| = ct$ and $|\overline{OR}| = pt$.  It follows that $w'({\bf x}(t))=(\sqrt{pc})t$.\end{proof}

\medskip

\section{An axiomatic characterization for the scale-invariant Hirsch citation index}  \label{SEC:Axiomatization}

Here we focus on Theorem \ref{T:char}, which provides an axiomatic characterization for the scale-invariant citation index $h'$.
Observe that all scale-invariant symmetric shape citation indices satisfy the properties of (Mon), (Sym), (SInv), (SSInv), (WResp), 
and (SqrtResp).  In Section \ref{SEC:LinearGrowth} we showed that $h$, $h'$, $w$, and $w'$ satisfy (LGr); similar arguments apply to all the scale-invariant symmetric shape citation indices.  The (MaxB) property also clearly holds for both $h$ and $h'$, and plays a fundamental role in selecting $h'$ alone, from the broader class of scale-invariant symmetric shape citation indices.

We start with two needed lemmas, the first being a result on sequences of positive real numbers.

\begin{lemma} \label{L:powerX}
Let $f\!\!:\mathbb{N} \rightarrow [1, + \infty)$  satisfy:
\begin{itemize}
\item[$(i)$] $f(1) = 1$, $f(2) = 2^a$ ($a>0$)
\item[$(ii)$] $f(mn)=f(m)f(n)$ (for $m>1$ and $n>1$)
\item[$(iii)$] $f(m) \leq f(n)$ whenever $m<n$.
\end{itemize}
Then, $f(x) = x^{a}$.
\end{lemma}

\begin{proof}
If not, let $p>2$ be the minimum integer such that $f(p) \neq p^a$. Then $p$ must be prime, else $p=m n$ with $m>1$ and $n>1$;
as $m,n<p$, $f(m) = m^a$ and $f(n) = n^a$ so that by condition $(ii)$, $f(p) = f(mn) = f(m)f(n)= m^an^a = p^a$, a contradiction. 
As $f(p) \neq p^a$, let $f(p) = (p+\epsilon)^a$ with $\epsilon \neq 0$. It is clear that  $f(p-1)=(p-1)^a$ because $p-1 <p$ and $f(p+1)=(p+1)^a$ because $p+1 = 2k$ with $k<p$. Thus, by condition $(iii)$ we know $|\epsilon| <1$.

For any $m>0$ there is a unique integer $b(m)$ such that
$$
p^m/2 < 2^{b(m)} < p^m < 2^{b(m)+1} < 2p^m.
$$
Assume first that  $\epsilon >0$. Choose $m_0 \in \mathbb{N}$ large enough to make $(1 + \frac{\epsilon}{p})^{m_0}>2$, whence
$$
(p+\epsilon)^{m_0} = p^{m_0} \left(1 + \frac{\epsilon}{p}\right)^{m_0}>2p^{m_0} > 2^{b(m_0)+1}.
$$
Thus, by conditions $(ii)$ and $(iii)$
$$
[f(p)]^{m_0} =f(p^{m_0}) \leq f(2^{b(m_0)+1}) = (2^{b(m_0)+1})^a < [(p + \epsilon)^{m_0}]^a = [(p + \epsilon)^{a}]^{m_0} = [f(p)]^{m_0},
$$
a contradiction.

Assume now $\epsilon <0$.  Choose $m_0 \in \mathbb{N}$ large enough to make $(1 + \frac{\epsilon}{p})^{m_0}< 1/2$, whence
$$
(p+\epsilon)^{m_0} = p^{m_0} \left(1 + \frac{\epsilon}{p}\right)^{m_0}<p^{m_0}/2 < 2^{b(m_0)}.
$$
Thus, by conditions $(ii)$ and $(iii)$
$$
[f(p)]^{m_0} =f(p^{m_0}) \geq f(2^{b(m_0)}) = (2^{b(m_0)})^a > [(p + \epsilon)^{m_0}]^a = [(p + \epsilon)^{a}]^{m_0} = [f(p)]^{m_0},
$$
a contradiction.

Thus,  $f(x) = x^a$ for all $x \in \mathbb{N}$.
%
\end{proof}

\begin{lemma}\label{LSR}
If the Hirsch power index $h'_a$ satisfies Linear Growth then $a = \frac{1}{2}$, so that $h'_a = h'$.
\end{lemma}

\begin{proof}  In Proposition~\ref{LGrProp} in Section \ref{SEC:LinearGrowth} we show that the scale-invariant Hirsch index $h' = h'_{\frac{1}{2}}$ satisfies Linear Growth (with slope greater than zero when parameters $p$ and $c$ are each strictly positive integers).   As $h'_a = (h')^{2a}$ it follows immediately that $h'_a$ \emph{fails} Linear Growth for $a \neq \frac{1}{2}$.  \end{proof}


Our main theorem comes in three forms.  The first axiomatizes the parameterized family of all scale-invariant Hirsch powers $\{h_a\}_{a > 0} = \{(h')^{2a}\}_{a > 0}$.  The Main version characterizes the scale-invariant Hirsch index $h'$, by adding Linear Growth to the list of axioms used in the first version.  The Strong version provides an alternate characterization of $h'$ by substituting Square Root Responsiveness for the combination of Weak  Responsiveness and  Linear growth.  The complementary virtues of the Main and Strong versions were discussed at the end of Section \ref{SomeAxioms}.  Note also that IIA for Citation Indices could replace Max-Bounded in any of the three characterizations, thanks to Proposition \ref{IIAandMaxB}.
\begin{theorem}[\bf Characterization Theorem] \label{T:char}
Let $g\!\!:X \rightarrow [1, + \infty)$ be a scientific citation index 
satisfying
Monotonicity,
Symmetry,
 Strong Scale Invariance,
  and Max-Bounded
\renewcommand{\labelenumi}{\roman{enumi}}
\begin{enumerate}
\item[$(i)$] \emph{[The {\bf Scale invariant Hirsch powers characterization}]} If $g$ additionally satisfies Weak Responsiveness,
 then $g =  h'_a$ for some $a > 0$.

\item[$(ii)$] \emph{[The {\bf Main scale invariant Hirsch characterization}]} If $g$ additionally satisfies both Weak Responsiveness 
and Linear Growth, then $g =  h' = h'_{\frac{1}{2}}$,

\item[$(iii)$] \emph{[The {\bf Strong scale invariant Hirsch characterization}]} If $g$ additionally satisfies Square Root Responsiveness,
 then  $g =  h'= h'_{\frac{1}{2}}$.\footnote{With no additional assumption of Linear Growth.}

\end{enumerate}
\end{theorem}

%
%
%

\begin{proof} [of $(i)$]
By Strong Scale Invariance with $k =2, m = 1$, $g(1)g(2,2) = g(1,1)g(2)$, and by Symmetry $g(1,1) = g(2)$, so that $g(1)g(2,2) = g(2)^2$.  But $g(1) = 1$ for every scientific citation index (Definition \ref{D:sci}), and $g(2,2) > 1$ by Weak Responsiveness, so $g(2)^2 > 1$, whence $g(2) > 1$. Choose $a > 0$ with $g(2) = 2^a$ ; we'll show $g = h_a$.  Consider first the restriction of $g$ to the subdomain consisting of all citation records $(k)$ that record a single paper with $k$ citations. 
This restriction of $g$ can be considered a function from $\mathbb{N}$ to $[1, + \infty)$, as in Lemma \ref{L:powerX}.

As $g$ satisfies $g(m)g(n) = g(mn)g(1)$ by (SSInv),  $g(mn)=g(m)g(n)$ holds for all $m>1$ and $n>1$.  Also,  $g(m) \leq g(n)$ holds whenever $m<n$ by (Mon), so we conclude that $g$ verifies the three requirements of Lemma~\ref{L:powerX}. Thus, $g(x) =x^a$ holds for all $x \in \mathbb{N}$, and it follows that this restriction of $g$ to the subdomain formed by citation records of a single paper is uniquely determined and coincides with the corresponding restriction of $h_a$.

Next, consider $g$ over the subdomain of citation records of form ${\bf m}_n = (\overbrace{m,\dots,m}^n)$---constant vectors that represent an arbitrary number $n$ of published papers, each with the same number $m$ of citations.  By (SSInv) we have
$g({\bf m}_n) g(1) = g({\bf 1}_n) g(m)$, where $g(1)=1$ and $g({\bf 1}_n) = g(n)$ by (Sym). Thus,
$g({\bf m}_n) = g(m)g(n)$, which determines $g$ on this larger subdomain.

Finally, consider $g$ over the full domain $X$ of citation records.
Let ${\bf x} = (x_1,x_2, \dots,x_{l({\bf x})}) \in X$ 
and let $l$ stand for $l({\bf x})$ in the next three equations. 
Let ${\bf x_i}_i = (\underbrace{x_i,\dots,x_i}_i)$ for $i=1,2, \dots,l.$
As ${\bf x_i}_i  \preceq  {\bf x}$, by (Mon) it follows that
\begin{equation} \label{EQ:FirstEquality}
g({\bf x_i}_i) \leq g({\bf x}),  i=1,2, \dots,l.
\end{equation}

\noindent By (MaxB) we have
\begin{equation} \label{EQ:Equality}
g({\bf x} ) \leq \max  \{ g({\bf x_1}_1), g({\bf x_2}_2), \dots, g({\bf x_l}_l)\},
\end{equation}
so that inequalities ~\eqref{EQ:FirstEquality} and ~\eqref{EQ:Equality} together yield
\begin{equation} \label{EQ:EqEquality}
g({\bf x} ) = \max  \{ g({\bf x_1}_1), g({\bf x_2}_2), \dots, g({\bf x_l}_l)\}.
\end{equation}
As
$g({\bf x_i}_i) = g(x_i)g(i)$ for all $i$, it follows that
$$
g({\bf x} ) =  \max_{j=1,\dots,l({\bf x})} \ \{  g(x_j)g(j) \}.
$$
Finally, as
$g(m)g(n)=g(m\, n)$ we deduce
$$
g({\bf x} ) =  \max_{j=1,\dots,l({\bf x})} \ \{  g(j \cdot x_j) \},
$$
whence, by Lemma~\ref{L:powerX},
$$
g({\bf x} ) =  \max_{j=1,\dots,l({\bf x})} \ \  \{ (j \, x_j)^{a} \} = h_a({\bf x}).
$$
This completes the proof of part $(i)$.  Part $(iii)$ follows immediately, and part $(ii)$ follows from $(i)$ together with Lemma \ref{LSR}.  \end{proof}

We conclude the section by discussing independence of our axioms.  The main result here is:

\begin{proposition}  \label{AxInd}
The five axioms (Mon), (Sym), (MB), (SSInv), 
\begin{enumerate}
\item[$1.$]
and (SqrtResp) that uniquely characterize the scientific citation index $h'$ (Part $(iii)$ of Theorem \ref{T:char}) are independent.
\item[$2.$]
and (WResp) that uniquely characterize the scientific citation powers $h_a$ (Part $(i)$ of Theorem \ref{T:char}) are independent.
\end{enumerate}
\end{proposition}

\begin{proof}
\begin{enumerate}
\item[$1.$]
\begin{itemize}
\item $h_{a}$ for $a>0$ with $a \neq \frac{1}{2}$ violates (SqrtResp) but satisfies (Mon), (Sym), (MaxB) and (SSInv) (as well as Scale Responsiveness).

\item The following index $t_\frac{1}{2}$ satisfies (Sym), (MaxB), (SSInv) and (SqrtResp), but violates (Mon).  For integers $j, k$ let ``$j \neq \dot{k}$" stand for ``$j$ is not an integer multiple of $k$."  We first define $t_\frac{1}{2}$ on the restricted domain of one paper with several citations:
$$
t_\frac{1}{2}(x) = \left\{
           \begin{array}{ll}
             x^\frac{1}{2}, & \hbox{if} \; x \neq \dot{3} \\
             (y \cdot 5^m)^\frac{1}{2}, & \hbox{if} \; x=y \cdot 3^m \; \, y \neq \dot{3} \; \, \text{and} \; \, m \geq 1
           \end{array}
         \right.
$$
The values of $t_\frac{1}{2}$ are now fixed by (Sym) on all vectors of  type ${\bf 1}_n$ for $n \in \mathbb{N}$, with (SSInv) then determining  values on constant vectors ${\bf m}_n$ for all $m,n \in \mathbb{N}$. Finally, (MaxB) with equality fixes the values of $t_\frac{1}{2}$ for all ${\bf x} \in X$. (SqrtResp) is obviously fulfilled.

The index violates (Mon) because $t_\frac{1}{2}(3) = \sqrt{5} > 2= t_\frac{1}{2}(4) $.


\item For an index $d$ satisfying (Mon), (MaxB), (SSInv), and (SqrtResp) but not (Sym) let $b$ be any strictly positive real constant with $b \neq  \frac{1}{2}$.  Set $d(n)=n^{\frac{1}{2}},$ $d({\bf 1}_m ) = m^b$, and $d({\bf n}_m ) = (n^\frac{1}{2})(m^b)$.
Finally, apply (MaxB) with equality to extend the values of $d$ over all ${\bf x} \in X$.

\item The scale-invariant Woeginger index $w'$ satisfies (Mon), (Sym), (SSInv), (SqrtResp) but violates (MaxB): Let ${\bf x} = (4,4)$, ${\bf y} = (2,2,2,2)$ and ${\bf z} = (4,4,2,2)$. Then $w'({\bf x}) = \sqrt{8} = w'({\bf y})$, but $w'({\bf z}) = \sqrt{16}=4.$ Thus, $w'({\bf z}) > \max \{ w'({\bf x}),  w'({\bf y}) \}$.

\item The citation index $f$ for which $f(1) = 1$ and $f({\bf x})=\sqrt{2}$ for all ${\bf x} \neq (1)$ violates (SSInv) since $\sqrt{2} =f(4) f(1) \neq (f(2))^2 = 2$, but trivially satisfies (Mon), (Sym), (MaxB), and (SqrtResp).
\end{itemize}
\item[$2.$] Indices $t_{\frac 12}$, $d$, $w'$ and $f$ satisfy (WResp), so most independencies follow as in part 1. The index $g \equiv 1$ satisfies (Mon), (Sym), (MaxB) and (SSInv) and violates (WResp). 
\end{enumerate} \end{proof}




Other dependencies include, of course, that (SInv) follows from (SSInv), while Lemma \ref{LSR} shows that in the presence of (Mon), (Sym), (MaxB), and (SSInv), (SqrtResp) is equivalent to (LGr) + (SResp). The reader will notice that Proposition \ref{AxInd} omits any claim of independence for the axioms used in part $(ii)$  of Theorem \ref{T:char}.
This is because Linear Growth axiom (LGr) would seem to have strong structural implications.  In particular, we do not know whether the presence of (LGr) allows any of the other part $(ii)$ axioms to be relaxed or even dropped.

\section{Simulations with Poisson Noise}\label{SEC:Simulations}

The actual rate at which an author publishes papers is subject to a variety of unpredictable factors, many of which are not under her control, and have little to do with the intrinsic quality of her work. Was the paper assigned to a demanding referee, or a lenient one?  Did the journal have a large backlog at the time?  Did a coauthor insist that two related papers be combined?  Similar factors apply to the rate at which an already published paper gathers citations, if only because these factors apply to the publication of  papers, by others, that are providing the citations.

Here we incorporate these features by viewing the accrual of publications and citations as a random process that can be modeled by ``Monte Carlo" simulations.  We were motivated by a suspicion that such simulations might reveal systematic differences between the indices of Hirsch and Woeginger in their original form, and their scale invariant versions.  In particular our initial thought was that the scale invariant versions might be more robust under noisy conditions---less likely to be knocked off track, because they respond more flexibly to change.  For example, there might be more opportunities for a single new citation (to a randomly chosen publication) to add to the area of the largest inscribed rectangle than would be the case for the largest inscribed square.

Somewhat later, we were looking at a list of Hirsch index values for all academic researchers in Spain (see \cite{WebSpain})
and were struck by the large number of researchers with identical index values. For instance, $1258$ researchers from Spain have an $h$-index equal to $23$ and $2220$ have an $h$ index of $16$. There are even surprisingly many ties among the $5000$ most cited researchers in the world \cite{WebTopWorld}, with up to $190$ scholars having an $h$-index of $100$.
We wondered whether the scale invariant versions might offer improved \emph{resolution}, with fewer ties (for reasons similar to those mentioned above).  Of course, in noisy contexts fewer ties might not be desirable; if ties are being broken primarily by the noise, then fewer ties might reflect \emph{false precision}.

Our simulations here are based on some admittedly very simple  (one might say simplistic) assumptions.  We assume that for each time increment both the number of new publications, and the number of citations  of each previously published paper, follow Poisson distributions (mass functions) that remain constant over the author's career. In effect, we add Poisson noise to the simple deterministic model. The Poisson was selected because it is a probability mass function (i.e., produces integers), allowing zeros but not negative numbers, and has no theoretical maximum value. Among probability mass functions satisfying these criteria, it is perhaps the most commonly applied in practice. The underlying noise model, according to Wikipedia, is that Poisson ``\emph{expresses the probability of a given number of events occurring in a fixed interval of time or space if these events occur with a known constant mean rate and independently of the time since the last event} \dots"
Neither our use of the simple deterministic model as base, nor the use of Poisson (with its independence assumption) would seem to closely mirror reality.  Nevertheless, we feel this model to be appropriate as a simple starting point from which future simulations can build.

Table 1 (Figure \ref{Table1and2}) shows the results of our simulations for three parameter value combinations, each specified on a \emph{monthly} basis:
\begin{itemize}
\item $p = 0.125$; $c = 0.32$
\item $p = 0.2$; $c = 0.2$
\item $p = 0.32$; $c = 0.125$
\end{itemize}

\begin{figure}[!ht]
{\centering
\includegraphics[height=105mm,width=170mm] {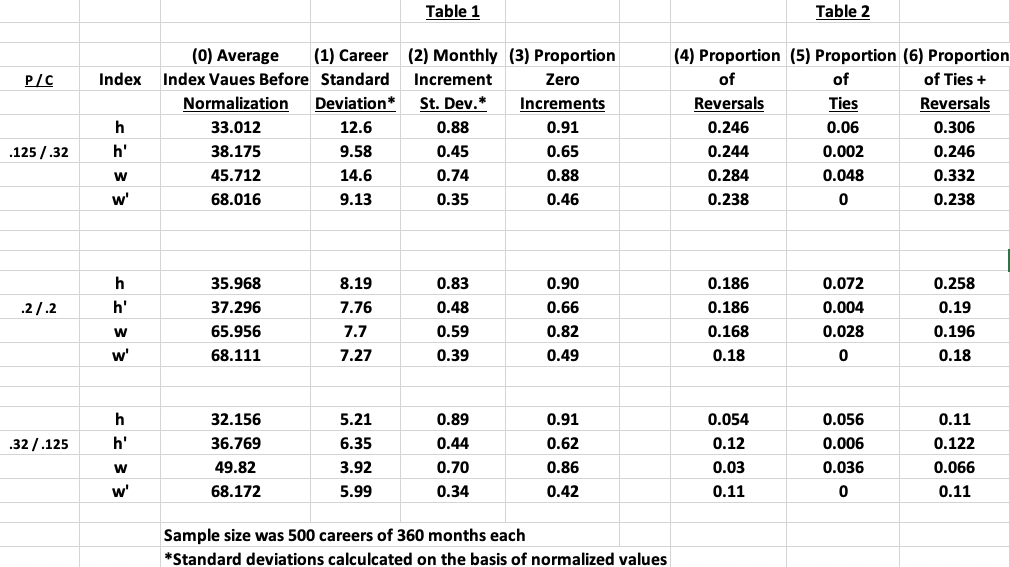}
\caption{Stochastic Simulations.}\label{Table1and2}}
\end{figure}


\noindent For each combination we randomly generated $500$ thirty-year careers.  That is, for each career and in each of $360$ consecutive months in that career, a non-negative integer was randomly generated from a Poisson distribution with mean $p$; this became the number of new publications for the author that month. Similarly, for each previous
publication, separately and independently a non-negative integer was randomly generated from a Poisson distribution with mean $c$; this became the number of citations for that paper that month.\footnote{Unlike the version we used originally for the standard deterministic model, in this section the papers published within a given time increment do not begin to accrue citations until the \emph{following} time increment.  (It seemed unrealistic, given the shorter monthly increments used here, for a paper to both appear and be cited by another paper within a single increment.)  As a check, we did run these simulations a second time, allowing citations to occur during the month of publications and, as expected, this made very little difference. }
We then updated the values of each of the four indices $h$, $h'$, $w$, and $w'$ for that month.

Note that a monthly publication mean of
  $0.125$ corresponds to an annual average number of $12\cdot 0.125 = 1.5$ publications, a monthly $0.2$ corresponds to an annual mean of $2.4$,  and a monthly $0.32$ yields an annual mean of $3.84$. We chose to update monthly, rather than annually, to obtain greater granularity in terms of the observed size of each step---each increment in an index value from one update to the next.  In particular, for these values of $p$ and $c$, most monthly index value increments are zero, as shown in column (3) of Table 1. Consequently, the increments we observe each month almost all represent single steps rather than the compound effect of several sequential increases in the index value.  We'll say more about this, shortly.

  Note, as well, that the product $p\cdot c$ has the same value $0.04$ for each of the three pairs of parameter values.  Consequently, if not for the presence of noise $h'$ would grow at exactly the same rate for each of these parameter pairs, as would $w'$  (based on the slopes as given in Proposition \ref{LGrProp}).  Thus, the three parameter pairs represent researchers of comparable strength, as measured by any scale-invariant index of the kind we are discussing.\footnote{The
  $h$-index, however, would grow about $10\% $ more quickly for $p = 0.2$,
 $c = 0.2$  than for the other two pairs, while $w$ would grow significantly more quickly for $p = 0.2$, $c = 0.2$ than for the other pairs.}

The matter of normalization for \emph{career values} (end-of-career values of an index for a researcher) required some care.  Two citation indices might as well be the same if one is a scalar multiple of the other, but when we measure the career standard deviation (standard deviation in end-of-career index values for the 500 runs) the same scalar will multiply those measurements.  Each of our four indices  produced a different average career value (over the 500 careers we ran), as seen in column (0) of Table 1, so if we measured standard deviation without compensating for these differences, the results would not be directly comparable.
To create an ``apples-to-apples"
comparison, for each index we divided each of the 500 separate career values by the column (0) average of the 500 career values for that index, and then multiplied by $100$ to obtain normalized career values that average 100 for each index.\footnote{The alternative, of normalizing separately for each single run of each index, is problematic.} These normalized career values were used in measuring career standard deviation.


The results in column (1) of Table 1 
show that for the citation dominant researcher ($c > p$), the scale invariant indices had lower standard deviation in their career values; they were more robust to the noise. The same holds for the balanced researcher  ($p = c$), but the reduction in standard deviation is much less dramatic. For the publication dominant researcher ($p > c$) the results reverse with the original indices being more robust. Taken overall, however, the scale invariant indices seem to offer greater robustness.  Note also that the standard deviations decreased from the first case (citation dominant) to the third case (publication dominant), suggesting that all four indices may be more sensitive to noise in citation rates than to noise in publication rates.

We calculated standard deviations of monthly increments as well (also normalized by average career index), which are shown in column (2) of Table 1. Note that the pattern for monthly increments is consistent across all three researchers; the scale-invariant indices produce less month-to-month variation, and this is loosely consistent with the information in column (3), showing that they have more frequent (hence, smaller after normalization) non-zero increments, compared to the original indices.

Table 2 (Figure \ref{Table1and2}) is meant to address the question of ties. Are some indices more resolute, and if so should we see greater resolution as advantageous?
For each of the same three combinations of $p$ and $c$ used earlier, we simulated 500 \emph{pairs} of careers; for each pair, the first researcher, who we will refer to as Researcher A, was assigned the nominal value of $p$ and of $c$,  while Researcher B was given $p$ and $c$ values $10\%$ higher than Researcher A. That is, if Researcher A had $p = 0.125, c = 0.32$, then Researcher B had $p = 0.1375$, $c = .352$. For each of the 500 pairs of careers, we compared the final career value for Researcher A to that of Researcher B for each index.  (We also ran a noise-free version, to show  the final values for A and B with no noise; of course, B receives the higher value in each case.)  A tie occurred when A and B received the same value (which we can think of as a form of ``wrong answer") with column (5) showing the proportion of ties.   Note that there were  a non-negligible proportion of ties, especially for the original indices.  It is perhaps unsurprising, in light of Table 1, that the scale invariant indices produced substantially fewer ties.

But were more ties being broken by the scale invariant indices simply because they were responding more sensitively to the noise? Let us say that a ``reversal" occurs when Researcher A's final career index is strictly higher than B's, i.e., when we get a ``very wrong answer." Column (4) of Table 2 shows that for the citation dominant and balanced researchers, the proportions of reversals were almost the same for the scale invariant indices as for the original versions.  For these cases, the additional tie-breaking ability of the scale invariant signals seems to reflect an enhanced ability to resolve the signal behind the noise, and not just a greater sensitivity to the noise itself.  Recall that real scholarly records in the data investigated by Fenner et al. \cite{FHLB} most frequently reflect the $c > p$ case, in which the scale-invariant indices achieve greater resolution without any increase in reversals.

For the publication dominant researchers, however, the pattern was similar to that for career standard deviation (in Table 1), with the original indices showing substantially fewer reversals. Column (6) represents the proportion of ties or reversals.  It is the sum of the previous two columns, and reflects the ``wrong or very wrong" answers---the proportion of times that the index failed to show that Researcher B was the better researcher. Again, it suggests that for citation dominant and balanced researchers, the scale-invariant versions out-perform their original counterparts.

\begin{figure}[!ht]
{\centering
\includegraphics[height=100mm,width=115mm] {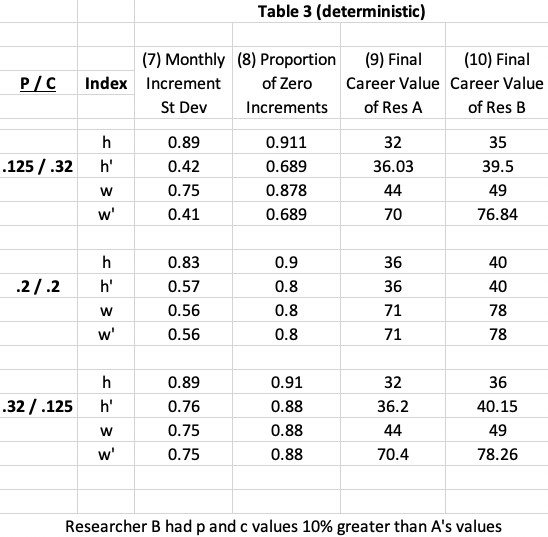}
\caption{Results with no Poisson noise.}\label{NoNoise}}
\end{figure}

For comparison purposes, Table
3 (Figure \ref{NoNoise}) shows the same simulations with no noise.  The calculation is deterministic, so there is no need for multiple runs, no column showing Career standard deviation (which would be zero), and no column for reversals (which would also be zero).
The monthly increments are not all equal, however, for two reasons.  First, even if the number of publications and citations increased by an identical integer amount for each time increment, these indices would occasionally jump up in value.  Second, we used non-integer values for $p$ and $c$, so with $p = 0.2$ (for example) each month adds $0.2$ publications to the total accrued to date. However, the indices only ``see" these increments once the accrued total reaches the next integer, at which point the new publication is recognized and can receive citations. Because the indices can experience jumps only when accrued publications or citations reach the next integer, the monthly increments vary for this reason as well.

A larger standard deviation in monthly increments implies fewer, more dramatic increments, i.e., steeper stair steps. Note that for all three combinations of $p$ and $c$, the variation was less for $h'$ than for $h,$ showing that the scale-invariant Hirsch index has more frequent and smaller increments---it is smoother than the Hirsch index. This is particularly true for $c > p,$ where the standard deviation of $h'$ is less than half that of $h$. For the case of Woeginger's index, however, $w'$ only has a lower monthly increment standard deviation than $w$ for $c > p$; in the other two cases, the variation is equal.

\section{Principal differences in behavior between the scale-invariant version $h'$ of Hirsch's index and the original index $h$}  \label{SEC:Goodness}

Principal differences in behavior between these two indices include:

\begin{enumerate}

\item \emph{$h'$ provides a fairer ranking within subdisciplines.}

\item \emph{Under $h'$ a greater proportion of a researcher's total citations contribute to the value of the index.}

\item \emph{$h'$ responds more smoothly to changes in the citation record, growing over time via increments that are more frequent and smaller.}

\item \emph{Under noisy conditions, $h'$ offers enhanced resolution and  decisiveness, with fewer ties and reversals.}

\end{enumerate}

\noindent While these points of difference are not entirely independent of one another, we do see them as four distinct advantages of $h'$ over $h$, worth articulating as separate points.  The comparisons are phrased in terms of $h'$ versus $h$, but similar observations apply to $w'$ versus $w$ (and to other, analogous pairs, we expect).
Of the four, a has already been discussed extensively in Section \ref{SEC:ScaleInvariantJustification}, and d in Section \ref{SEC:Simulations}.  Here, we discuss additional properties that provide support for b and c.

\bigskip

\noindent \emph{Tail balancedness and Centrality} Before the advent of indices such as Hirsch's, it was more common to use, as a metric,
the sum of the number of citations accrued by each publication, which of course is equal
to the total area under the step-function. By counting only that part of the area contained in a convex sub-region, Hirsch's index
lops off two tails so that only those citations lying inside Hirsch's square contribute to the value of the index.  Notice that the squares (and the isosceles triangles)  used by $h$ (and by $w$) are symmetric about
the line $y = x$. One consequence is that the tail truncation rule (implicit in either index) for horizontal tails is the same as that for vertical tails. It is worth noting, however,
that Hirsch's argument against counting the area in the horizontal tail is quite different from
that for the vertical tail. Hirsch does not provide theoretical or empirical evidence suggesting
these tails resemble one another in shape or area. We see no good argument, then, for treating the
two tails in exactly the same way.

Consider a citation record for which the horizontal tail is significantly larger in area than the vertical, with the step function remaining reasonably high over a significant range lying to the right of the Hirsch square.  Compared to the  square, the max-area $h'$ rectangle will lose height (as it retreats from the vertical tail) so that it can squeeze into the horizontal tail, gaining area because the gain in the rectangle's width more than compensates for the reduction in height. The two tails for the rectangle will now be closer to each other in area (\emph{tail balancedness}) and the max-area rectangle will be more centrally located (\emph{greater centrality}) within the bar graph region.  Of course, a similar argument applies if the vertical tail is the larger one. This is the mechanism by which $h'$ credits a larger proportion of an author's citations (point b above).  That it happens in practice was demonstrated by Fenner et al. \cite{FHLB}, who show $h' >> h$  often holds for real citation records (see discussion in Section \ref{SEC:INTRO}.1).

\bigskip

\noindent \emph{Finite-to-oneness} As a function from the set of all possible citation records, the scale-invariant index $h'$ is finite-to-one, while the original version is not.  To see why, suppose a certain scholar $R$ has a Hirsch index of $n$. Then we can easily infer that the lower bound of the number of citations of this scholar is $n^2$, but there is no upper bound. So in principle there are infinitely many citation records that yield the identical index value of $n$.  In practice, this allows for a possibly long list of actual scholars having identical index values of $n$, which includes scholars whose total number of citations varies widely.

The situation for $h'$ is quite different. If scholar $R$ has a scale-invariant Hirsch index of $x$, we can similarly infer that the lower bound of the number of citations of this scholar is $x^2$, but we can also find an upper bound, as given by
\begin{equation}  \label{E:upperbound}
\sum\limits_{i=1}^{x^2} \; \left\lfloor \dfrac{x^2}{i} \right\rfloor
\end{equation}
The sum in \eqref{E:upperbound} can be bounded by
\begin{equation} \label{E:upperbound2}
x^2+ \int_1^{x^2} \, \, \dfrac{x^2}{y} \, \, dy  = x^2 (1+2\ln x).
\end{equation}
For example, if $x=12$ the exact upper bound in \eqref{E:upperbound} is $746$, and the approximation in \eqref{E:upperbound2} gives $859$. Thus, the list of scholars having $h'=12$ contains researchers with at least $144$ citations but at most $746$ citations. 
This distinction may help explain why $h'$ credits a greater proportion of total citations, and also why $h'$ is more responsive to changes in the citation record.

\medskip

\noindent \emph{More finely divided range} The set of images under the Hirsch index contains the positive integers, while for the scale-invariant version the range contains all square roots of positive integers. Between two consecutive integers $n$ and $n+1$ there are $2n$ non-integer numbers whose squares are integers, and the number of potential values of the scale-invariant Hirsch index that lie under a given integer $n$ is $n^2$.

Suppose, for example, that the Hirsch indices of all scholars from a certain discipline are less than $100$.  As $h \leq h'$, some of these scholars may have an $h'$ value inferior to $100$, but others may have a greater index. For the first group of scholars there are $10,000$ possible images, and of course there are other images available from the second group. So the number of available $h'$ values is greater than the square of the number of available $h$ values.  This offers the possibility for $h'$ to respond more smoothly to changes in the scholarly record, and to have fewer ties (points c and d) and may explain why we observe that behavior in the simulations of the previous section.

\medskip

\noindent \emph{Flexible options for increments} Suppose that a researcher's $h$-index currently has value $n$, and a single new citation is added to her citation record. This can result in an increment to the $h$-index only if the new citation is made to (one of) the $n+1$-th most cited articles.  (There are additional conditions: each of the $n$ most cited papers must have at least $n+1$ citations and the newly cited paper must have exactly $n$ citations.  In this case, the Hirsch square simultaneously increases by 1 in both height and width, and the index increases from $n$ to $n+1$.)

With $h'$, however, a new citation to any of the $i=1, 2, \dots, l({\bf x})$ papers can potentially increase the $h'$-index (with additional conditions attached to each possibility) as the max-area inscribed rectangle increases slightly in height (while the width stays the same), or increases a lot in height (while the width decreases), or increases slightly in width (while the height stays the same), or increases a lot in width (while the height decreases), or increases in both height and width.  Recall, as well, that the size of the inscribed Hirsch square can be (and, arguably, often would be) constrained by only a single point of contact with the correspondence $c_{{\bf x}}$, so that any increase in the side-length of the square requires a change to the citation record that affects $c_{{\bf x}}$ at that particular location.  In contrast, the max-area inscribed rectangle for $h'$ is constrained by at least two points of contact with $c_{{\bf x}}$.  Arguably, there often would be exactly two such contact points, in which case a change to the citation record that affects $c_{{\bf x}}$ at either location would allow the rectangle to grow.  These differences may explain why we found, for both the stochastic and the deterministic simulations in Section \ref{SEC:Simulations}, that the scale invariant version of the index changes more frequently over the course of a career, resulting in a smoother response by the index (point c on the list).

\medskip

\bigskip

\section{Conclusions and directions for future research \label{SEC:MoreQuestions}}

The scope of our work here has been rather narrow.  Some citation indices are designed to score more highly those records that strike a balance between productivity and impact.  Many of these, including the ones proposed by Hirsch and by Woeginger, have some implicit fixed ratio of productivity to impact built into their underlying notion of balance.  We contend that such fixed scale indices introduce distortions in their rankings of researchers, while modified versions that impose scale invariance avoid those distortions.  Moreover, the scale invariant versions have some attractive axiomatic properties that argue for their mathematical naturality, tend to produce fewer ties, and to be more resistant to noise. In short, we argue that if one chooses to employ a \emph{balanced} index, then it should be a scale invariant version.

The mathematical social sciences are rife, however, with examples of desirable properties for social mechanisms that are inconsistent with one another, Arrow's Impossibility Theorem being a particularly well-known example. It has not been our purpose, here, to weigh in on whether the goal of rewarding \emph{balance} is worth the cost of giving up on other principles incompatible with that goal, such as independence or batching consistency (see Section \ref{BalanceCost}).  Nor have we addressed a host of other issues (such as how to factor in the effect of article length, journal quality, or co-authorship with others whose index values might be much higher or lower) that must be confronted when constructing a practical instrument that could be credibly applied in the real world.\footnote{
Some of these other issues might be best addressed via preliminary adjustments to the citation record, rather than by changes to the index itself. For example, one might adjust the number of citations credited to a paper according to the number of co-authors.  At the fifth World Congress of the Game Theory Society (in Maastricht, 2016), in the same session in which we presented a preliminary version of this paper, a presentation by Karol Flores-Szwagrzak and
Rafael Treibich proposed a fixed point mechanism for adjusting the number of citations credited to a paper according to the citation records of one's co-authors; see \cite{FT2020}). If the citation index itself were then applied to an adjusted record, we might view the question of what adjustments to make as being largely separable from the question of which index to apply after adjustment.}  

Our work here suggests a variety of directions for future research.  First, which practical issues would need to be addressed before any scale invariant index such as $h'$ could be implemented responsibly?  We have already mentioned that in its unadulterated form a scale-invariant index such as $h'$ does not actually reward balance; the greatest area rectangle might arise from a single highly cited paper, or from many papers, each with at least one citation.  John Nash himself presents an interesting case, which is almost this extreme.  According to Google Scholar, his third most cited paper has 10,675 citations, after which the numbers of citations drop off precipitously, so that the maximal area rectangle has dimensions $ 3 \times 10,675$, with $h' = 178.96$.  According to Google Scholar, Nash's Hirsch index is only $14$.\footnote{But the actual value may be lower, as his Google Scholar page listed a paper that is \emph{about} Nash, but not written by him, and lists a reprint of his paper on the bargaining problem, along with the original.} 

We presume that upper and lower \emph{proportion bounds} would be placed on the ratio of height to width of the maximal area rectangle; only rectangles whose proportions fall within those bounds would then be considered in calculating the correspondingly restricted version of $h'$.  What should those limits be?  Notice that there probably exist no citation records as extreme as Nash's, that have the inverse proportions (that is, having over $10,000$ publications, many of which have as few as $3$ citations); this suggests that the lower proportion bound might be quite unequal to the inverse of the upper one.  Should these bounds be the same for all fields, or should they vary even among subfields of a given field?  For what (hopefully large) fraction of active scholars would the restrictions have no effect (because the best rectangle already falls within the bounds)?  

Those decisions should probably rest on a better understanding of the range in proportions of the max-area rectangle: within various subfields, from subfield to subfield, and from field to field.  Studies comparing these ranges may be of independent bibliometric interest.  Such a bibliometric study might suggest one or more additional parameters based on discipline, reflecting the typical range of dimensions for the max area rectangles of researchers from that discipline, the number of active scholars in the discipline, etc. 
Such parameters might suggest the possibility of some more uniform rule---one that would set the proportion bounds for a discipline according to the parameter values for that discipline.  They might also suggest how to construct correction factors that could be applied to compensate for disciplinary differences, allowing for fairer comparisons of scholars from different disciplines or sub-disciplines.     

In terms of our axiomatic results, the most obvious gap is that we have no  characterization for the scale invariant version $w'$ of Woeginger's index.  One goal in particular should be to find a substitute to the Max Bounded axiom that would convert the Main Characterization Theorem \ref{T:char}\emph{(ii)} for $h'$ into a corresponding result for $w'$ (or, with Linear Growth dropped, yield a characterization for the class of Woeginger powers $w_a$ for $a>0$).  On the other hand, if we simply drop Max Bounded from the list of Theorem \ref{T:char}\emph{(ii)} axioms, the class of indices so characterized includes all the scale invariant symmetric shape indices of Definition \ref{ShapeIndex} (including $h'$ and $w'$), as well as indices based on several different shapes, for which the final value is based on the shape that fits best (in the sense of yielding the highest value for the given ${\bf x}$).  Can we pin down a structural characterization for that class?

Given that the original Hirsch index $h$ satisfies all the axioms of Theorem \ref{T:char}\emph{(ii)} {\em except} for Strong Scale Invariance, it would also be of great interest to characterize $h$ by substituting some alternative to that axiom, perhaps one that represents a type of denial of scale invariance.  As one example of such a denial, note that $h$  (but not $h'$) satisfies the following requirement: for each choice of length $\ell$ there is a positive integer $k$ such that $g(k'{\bf x}) =  g(k{\bf x})$ holds whenever $\ell ({\bf x}) = \ell$ and $k' > k$ is an integer.

The inclusion of the Linear Growth axiom (LGr) in Theorem~\ref{T:char}\emph{(ii)} made it difficult to know whether or not the axioms used in this part are independent.  A related question is whether, in the presence of (LGr),  any of the other axioms used in this part could be relaxed. For example, can strong scale invariance (SSInv) be replaced by scale invariance (SInv)? Any progress on refining Theorem \ref{T:char}\emph{(ii)} would be valuable.

Finally, in terms of our simulation with noise, the most compelling open question is whether our results would broadly hold up under different models. One might explore alternatives to Poisson noise.  Alternately (or additionally) it is tempting to consider alternatives to the simple deterministic model as the pre-noise base.  Our model here assumed that for a given researcher $R$, all of $R$'s published papers earn citations at the same (pre-noise) rate, and that this rate remains constant over the years. In reality, some of $R$'s papers may have fundamentally greater impact than others, and the natural lifespan of a paper may see its yearly citations rise for some time, and then fall.  We remain curious, as well, about the underlying explanations for some of our results; why, for example, do the $h$ vs $h'$ and $w$ vs $w'$ results vary according to whether $c > p$ or $p > c$?

\section*{Acknowledgements}

The first author's research was partially supported by funds from the Ministry of Science and Innovation grant PID2019-104987GB-I00.  We thank Denis Bouyssou for his informative comments, which improved the manuscript.  

\section{Appendix: MVIIA, WARP, and Replacement}

The MVIIA axiom (Multi-Valued Independence of Irrelevant Alternatives Axiom, from Section \ref{SomeAxioms}) is related to two other principles from the mathematical social sciences.  First, we'll show that it implies the  \emph{Weak Axiom of Revealed Preference}, aka WARP---a condition on the choices made by an agent (from various sets of alternatives), which is  satisfied if those choices are ``rational," meaning they are guided by an underlying weak preference order over the alternatives.  Then we will show that MVIIA implies a related principle MVIIA$^\star$, which resembles the  \emph{reinforcement} principle of voting theory, used by Smith \cite{Smith} and Young \cite{Young}, \cite{Young2} to characterize scoring rules. 

Our original formulation of MVIIA was designed for a narrow context, in which the choice function selected points specifically from bar graph regions of  $\mathbb{R}_{++} ^2$. Nothing prevents us, however, from reformulating the same principle more abstractly, and this facilitates comparisons to principles from different contexts.

\begin{definition}\label{AbstractMVIIA}

Let $\Sigma$ be a collection of nonempty subsets of some set $X$, and $c$ be a choice function on $\Sigma$, meaning that $c$ selects a nonempty subset $c(F) \subseteq F$ for each set $F \in \Sigma$. Then $c$ satisfies the \emph{abstract version of MVIIA} if $c(G) = c(F) \cap G$ holds whenever $F,G \in \Sigma$ satisfy both $G \subseteq F$  and $G \cap c(F) \ne \emptyset$.

\end{definition}

\begin{definition}\label{SigmaClosure}

A collection $\Sigma$ of nonempty subsets of some set $X$ is \emph{closed under unions} if $F \cup G \in \Sigma$ holds whenever $F,G \in \Sigma$; $\Sigma$ is \emph{closed under intersections} if $F \cap G \in \Sigma$ holds whenever $F,G \in \Sigma$.

\end{definition}

We will need closure of $\Sigma$ under unions for the first proposition below, and closure under intersections for the second, so it is worth noting that the collection of all bar graph regions of $\mathbb{R}_{++} ^2$ has both closure properties.  The reason is that for any two citation records ${\bf z}, {\bf w}$, we have $\mathcal{B}({\bf z}) \cup \mathcal{B}({\bf w}) = \mathcal{B}(max({\bf z}, {\bf w}))$, which is also a bar graph region.  Here $max({\bf z}, {\bf w})$ refers to the componentwise maximum of the two citation records (which is itself a citation record).  Similarly $\mathcal{B}({\bf z}) \cap \mathcal{B}({\bf w}) = \mathcal{B}(min({\bf z}, {\bf w}))$, which is also a bar graph region.

The Weak Axiom of Revealed Preferences similarly refers to a collection $\Sigma$ of nonempty subsets of a set $X$, along with choice function $c$ on $\Sigma$. 
But the actual WARP statement (from \cite{ChamEch}, page 19) is posed in terms of derived relations $\succeq^c$ and $\nsucc^c$, as follows:

\begin{equation} \label{E:WARP}
x \succeq^c y \Rightarrow y \nsucc^c x. 
\end{equation}

\noindent Here $x \succeq^c y$ holds if for some $B \in \Sigma$ we have $x,y \in B$ and $x \in c(B)$; this says that at least once, $x$ gets chosen when $y$ was available.  We write $x \succ^c y$ if for some $B \in \Sigma$ we have $x,y \in B$ and $x \in c(B)$ and $y\notin c(B)$;  this says that at least once, $x$ {\em gets chosen over }$y$ (meaning $x$ gets chosen and $y$ is {\em not} chosen when $y$ was available).  Equation \eqref{E:WARP} thus asserts that if it ever happens that $x$ is chosen when $y$ is available, then $y$ is never chosen over $x$.

\begin{proposition} (MVIIA implies WARP) \label{MVIIAvsWARP} 
Let $\Sigma$ be a collection of nonempty subsets of some set $X$, closed under unions, and $c$ be any choice function on $\Sigma$ (in the sense of Definition \ref{AbstractMVIIA}). If $c$ satisfies the abstract version of  MVIIA, then $c$ satisfies WARP.

\end{proposition}

\begin{proof}
Assume the multi-valued choice function $c$ satisfies MVIIA. To show WARP, assume $x \succeq^c y$. Choose an $F \in \Sigma$ such that $x,y \in F$ with $x \in c(F)$. To show  $y \nsucc^c x$, let $G$ be any set in $\Sigma$ such that $y \in c(G)$ and $x \in G$. We'll show $x \in c(G)$.
Let $H = F \cup G$. Then $H \in \Sigma$, by our closure assumption. As $c(H) \subseteq F \cup G$,  it must be that either $c(H) \cap F \ne \emptyset$ or $c(H) \cap G \ne \emptyset$.

\underline{Case 1}: Assume $c(H) \cap F \ne \emptyset$.  Then by MVIIA, 
$c(F) = F \cap c(H)$.  As $x \in c(F)$, $x \in c(H)$.  Also, $x \in G$, so $G \subseteq H$ with $G \cap c(H) \ne \emptyset$. By MVIIA again, $c(G) = G \cap c(H)$.  But $x \in G$ and $x \in c(H)$, so $x \in G \cap c(H)$, whence $x \in c(G)$, as desired.

\underline{Case 2}: Assume $c(H) \cap G \ne \emptyset$.  Then by MVIIA, 
$c(G) = G \cap c(H)$. As $y \in c(G)$, $y \in c(H)$.  Also, $y \in F$.  So $F \subseteq H$ with $F \cap c(H) \ne \emptyset$. By MVIIA again, $c(F) = F \cap c(H)$. As $x \in c(F)$, $x \in c(H)$.  So $x \in G$ and $x \in c(H)$. As $c(G) = G \cap c(H)$, $x \in c(G)$, as desired. \end{proof}

 The statement of  \emph{MVIIA} also seems reminiscent of the \emph{reinforcement axiom}, used by Smith \cite{Smith} and Young \cite{Young} in characterizing scoring rules as a subclass of all those voting rules that are both \emph{variable electorate} (meaning the same rule can be applied to different electorates) and \emph{irresolute} (meaning that ties can lead to more than one winner).  Consider a scenario in which the same election (meaning the same set $A$ of candidates and same voting rule $R$) is held in two districts $P$ and $Q$ that have no voters in common. Reinforcement asserts that if there is any candidate who is both a winner in district $P$ and a winner in district $Q$ then when we apply the same rule $R$ to the combined district $P+Q$, the winners should be exactly those candidates who were winners in both districts.  More formally, if $w(P) \cap w(Q) \ne \emptyset$, then $w(P+Q) = w(P) \cap w(Q)$.
 
 As far as we know, the difference in context rules out any direct logical connection between MVIIA and reinforcement.  However, there does exist a connection between MVIIA and the following principle \emph{MVIIA}$^\star$, which seems quite parallel in spirit to reinforcement:  
 
 \begin{definition}\label{MVIIAstar}
Let $\Sigma$ be a collection of nonempty subsets of some set $X$, closed under intersection, and let $c$ be a choice function on $\Sigma$, in the sense of Definition \ref{AbstractMVIIA}. Then $c$ satisfies MVIIA$^\star$ if for all $F,G \in \Sigma$ satisfying $c(F) \cap c(G) \ne \emptyset$, we have $ c(F\cap G) = c(F) \cap c(G)$. 
\end{definition}
 
\noindent That is, $c$ selects from
 $F\cap G$ those points that were selected in common from both $F$ and
 $G$, providing at least one such point was selected from both.  
 
 \begin{proposition}  \label{MVIIAvsMVIIAstar} (MVIIA implies MVIIA$^\star$)
 Let $\Sigma$ be a collection of nonempty subsets of some set $X$, closed under intersection, and let $c$ be a choice function on $\Sigma$, in the sense of Definition \ref{AbstractMVIIA}. If $c$ satisfies the abstract version of MVIIA then $c$ satisfies MVIIA$^\star$.
\end{proposition}

\begin{proof}
Given $F$, $G \in \Sigma$
with $c(F) \cap c(G) \neq \emptyset$, let $H=F\cap G$.
Then $c(F) \cap H = c( F)\cap (F \cap G) \supseteq c(F) \cap c(G) \cap (F \cap G)$ $= c (F) \cap c (G) \neq \emptyset$.  So, by \emph{MVIIA}, $c(H) = H \cap c(F)$, whence $c(F \cap G) = F \cap G \cap c(F)$.  Similarly, $c(F \cap G) = F \cap G \cap c(G)$.  So $c(F \cap G) = [F \cap G \cap c(F)] \cap [F \cap G \cap c(G)] = c(F) \cap c(G). $ \end{proof}

We do not know whether a reasonable converse exists to Proposition \ref{MVIIAvsWARP}, or to Proposition \ref{MVIIAvsMVIIAstar}. It seems likely that the closure conditions placed on $\Sigma$ would play a role in any consideration of such converses.


\begin{thebibliography}{10}

\bibitem{AdKo15}
T.~Adachi and T.~Kongo.
\newblock Further axiomatizations of Egghe's $g$-index.
\newblock {\em Journal of {I}nformetrics}, 9:839--844, 2015.

\bibitem{BoDa05}
L.~Bornmann and H.D.~Daniel.
\newblock Does the $h$-index for ranking of scientists really work?
\newblock {\em Scientometrics}, 65(3):391--392, 2005.

\bibitem{BoDa07}
L.~Bornmann and H.D.~Daniel.
\newblock What do we know about the $h$ index?
\newblock {\em Journal of the American Society for Information Science and Technology}, 58(9):1381--1385, 2007.

\bibitem{BoDa09}
L.~Bornmann and H.D.~Daniel.
\newblock The state of $h$ index research. Is the $h$ index the ideal way to measure research performance?
\newblock {\em {EMBO} Reports}, 10(1):2--6, 2009.


\bibitem{BM1}
D.~Bouyssou and T.~Marchant.
\newblock An axiomatic approach to bibliometric rankings and indices.
\newblock {\em Journal of {I}nformetrics}, 8:449--477, 2014.

\bibitem{BM2}
D.~Bouyssou and T.~Marchant.
\newblock Ranking scientists and departments in a
consistent manner
\newblock {\em Journal of the American Society for Information Science and Technology}, 62(9):1761--1769, 2011.

\bibitem{ChamEch}
C.~Chambers and F.~Echenique.
\newblock Revealed Preference Theory.
\newblock{\em Econometric Society Monographs}, Cambridge University Press, 2016.

\bibitem{Egg06}
L.~Egghe.
\newblock An improvement of the $h$-index: The $g$-index.
\newblock {\em {ISSI} Newsletter}, 8--9, 2006.


\bibitem{EPR35}
A.~Einstein, B.~Podolsky and N.~Rosen.
\newblock Can {Q}uantum-mechanical description of physical reality be considered complete?
\newblock {\em Physical {R}eview}, 47(10):777--780, 1935.

\bibitem{FHLB}
T.~Fenner, M.~Harris, M.~Levene, and J.~Bar-Ilan.
\newblock A novel bibliometric index with a simple geometric interpretation.
\newblock {\em PloS one}, 13(7): e0200098, 2018.

\bibitem{FT2020}
K.~Flores-Szwagrzak and R.~Treibich.
\newblock Teamwork and Individual productivity.
\newblock {\em Management Sciences}, 66(6):2523--2544, 2020.




\bibitem{Hir05}
J.E.~Hirsch.
\newblock An index to quantify an individual's scientific research output.
\newblock {\em Proceedings of the {N}ational {A}cademy of {S}ciences}, 102(46):16569--16572, 2005.

\bibitem{Hir07}
J.E.~Hirsch.
\newblock Does the $h$-index have predictive power?
\newblock {\em Proceedings of the {N}ational {A}cademy of {S}ciences}, 104(49):19193--19198, 2007.

\bibitem{Kal77}
E.~Kalai.
\newblock Proportional solutions to bargaining situations:  intertemporal utility comparisons.
\newblock {\em Econometrica}, 45(7):1623--1630, 1977.



\bibitem{KalSmo75}
E.~Kalai and M.~Smorodinsky.
\newblock Other solutions to Nash's bargaining problem.
\newblock {\em Econometrica}, 43(3):513--518, 1975.


\bibitem{Kon14}
T.~Kongo.
\newblock An alternative axiomatization of the Hirsch index.
\newblock {\em Journal of {I}nformetrics}, 8(1):252--258, 2014.

\bibitem{LJL08}
S.~Lehmann, A.~Jackson and B.~Lautrup.
\newblock A quantitative analysis of indicators of scientific performance.
\newblock {\em Scientometrics}, 76(2):369--390, 2008.

\bibitem{LFB2019}
M.~Levene, T.~Fenner, and J.~Bar-Ilan.
\newblock Characterisation of the $\chi$-index and the \emph{rec}-index.
\newblock {\em Scientometrics}, 120(2):885-896, 2019.

\bibitem{Ley09}
L.~Leydesdorff.
\newblock How are new citation-based journal indicators adding to the bibliometric toolbox?
\newblock {\em Journal of the {A}merican {S}ociety for {I}nformation {S}cience}, 60(7):1327--1336, 2009.

\bibitem{Marchant2009a}
T.~Marchant.
\newblock An axiomatic characterization of the ranking based on the h-index and some other bibliometric rankings of authors.
\newblock {\em Scientometrics}, 80(2):325--342, 2009.

\bibitem{Marchant2009b}
T.~Marchant.
\newblock Score-based bibliometric rankings of authors.
\newblock {\em Journal of the American Society for
Information Science and Technology}, 60:1132--1137, 2009.

\bibitem{Nas50}
J.F.~Nash.
\newblock The bargaining problem.
\newblock {\em Econometrica}, 18(2):155--162, 1950.

\bibitem{Salles18}
M.~Salles.
\newblock Independence of irrelevant alternatives: Arrow, Nash.
\newblock {\em Slide presentation}, Workshop surprise en l'honneur de Ragip Ege, University of Strasbourg,
21-22 September, 2018.

\bibitem{Salles23}
M.~Salles.
\newblock The possibility of generalized social choice functions and Nash's independence of irrelevant alternatives.
\newblock {\em Social Choice and Welfare}, 60:299--311, 2023.


\bibitem{Smith}
J.H.~Smith.
\newblock Aggregation of preferences with variable electorate.
\newblock {\em Econometrica}, 41(6):1027--1041, 1973.


\bibitem{WE1}
L.~Waltman and N.J.~van Eck.
\newblock A taxonomy of bibliometric performance indicators based on the
property of consistency. 
\newblock {\em Technical Report ERS-2009-014-LIS, Erasmus University Rotterdam, Erasmus Research Institute of Management, Rotterdam, the Netherlands}, Presented
at the 12th International Conference on Scientometrics and Informetrics, Rio de Janeiro, July
2009.
 
 
\bibitem{WE2}
L.~Waltman and N.J.~van Eck.
\newblock The inconsistency of the h-index. 
\newblock {\em Journal of the American
Society for Information Science and Technology}, 63(2):406--415, 2012.

\bibitem{WE3}
L.~Waltman and N.J.~van Eck.
\newblock A taxonomy of bibliometric performance indicators based on the property of consistency. 
\newblock {\em Technical report, ERIM, 2009},
http://publishing.eur.nl/ir/repub/asset/15182/ERS-2009-014-LIS.pdf, 2009.

\bibitem{Woe08a}
G.J.~Woeginger.
\newblock An axiomatic characterization of the Hirsch-index.
\newblock {\em Mathematical {S}ocial {S}ciences}, 56(2):224--232, 2008.

\bibitem{Woe08b}
G.J.~Woeginger.
\newblock A symmetry axiom for scientific impact indices.
\newblock {\em Journal of {I}nformetrics}, 2(4):298--303, 2008.


\bibitem{WebSpain}
\newblock Webometrics: Ranking of researchers in Spain and Spaniards abroad.
\newline
\url{https://www.webometrics.info/en/GoogleScholar/Spain}

\bibitem{WebTopWorld}
\newblock Webometrics: Highly cited researchers ($h>100$) according to their Google Scholar Citations public profiles.
\url{https://www.webometrics.info/en/hlargerthan100}

\bibitem{Young}
H.P.~Young.
\newblock A note on preference aggregation.
\newblock {\em Econometrica}, 42(6):1129--1131, 1974.

\bibitem{Young2}
H.P.~Young.
\newblock Social choice scoring functions.
\newblock {\em SIAM J. Appl. Math}, 28(4):824--838, 1975.

%



\end{thebibliography}
\end{document}